\documentclass[reqno, 11pt]{amsart}

\usepackage{amsmath}
\usepackage{amsthm}
\usepackage{amsfonts}
\usepackage{amssymb}
\usepackage{mathrsfs}
\usepackage{epsfig}
\usepackage{slashed}
\usepackage{mathtools}
\usepackage{enumerate}
\usepackage[top=1in,bottom=1in,left=1.25in,right=1.25in]{geometry}
\usepackage{xcolor}

\newcommand{\la}{\lambda}
\newcommand{\mb}{\mathbf}
\newcommand{\mc}{\mathcal}

\renewcommand{\Im}{\mathrm{Im}\,}

\newcommand{\N}{\mathbb{N}}
\newcommand{\R}{\mathbb{R}}
\newcommand{\C}{\mathbb{C}}
\newcommand{\Z}{\mathbb{Z}}
\newcommand{\B}{\mathbb{B}}

\newcommand{\rst}[1]{\ensuremath{{\mathbin |}%
\raise-.5ex\hbox{$#1$}}} 

\newcommand{\norm}[1]{{\left\vert\kern-0.25ex\left\vert\kern-0.25ex\left\vert #1 
    \right\vert\kern-0.25ex\right\vert\kern-0.25ex\right\vert}}

\newtheorem{lemma}{Lemma}[section]
\newtheorem{theorem}[lemma]{Theorem}
\newtheorem{corollary}[lemma]{Corollary}
\newtheorem{proposition}[lemma]{Proposition}
\theoremstyle{remark}
\newtheorem{remark}[lemma]{Remark}

\theoremstyle{definition}

\numberwithin{equation}{section}

\title[]{Nonlinear stability of homothetically shrinking Yang-Mills solitons in the equivariant case}
\author{Irfan Glogi\'c}
\address{Universitat Wien, Fakult\"at f\"ur Mathematik, Oskar-Morgenstern-Platz 1, 1090 Vienna, Austria}
\email{irfan.glogic@univie.ac.at}

\author{Birgit Sch\"orkhuber}
\address{Karlsruher Institut f\"ur Technologie, Fakult\"at f\"ur Mathematik,  Englerstra\ss e 2, 76131 Karlsruhe, Germany}
\email{birgit.schoerkhuber@kit.edu}

\thanks{Irfan Glogi\'c acknowledges the support of the Austrian Science Fund FWF, Project P 30076. \\
Funded by the Deutsche Forschungsgemeinschaft (DFG, German Research Foundation) - Project-ID 258734477 - SFB 1173}

\usepackage{color}
\usepackage{kantlipsum}
\usepackage[colorlinks=true,linkcolor=blue,citecolor=red,urlcolor=blue]{hyperref}

\begin{document}
\begin{abstract}
We study the heat flow for Yang-Mills connections on $\R^d \times SO(d)$. It is well-known that in dimensions $5 \leq d \leq 9$ this model admits
homothetically shrinking solitons, i.e., self-similar blowup solutions, with an explicit example given by Weinkove \cite{Wei04}. We prove the nonlinear asymptotic stability of the Weinkove solution under small equivariant perturbations and thus extend a result by the second author and Donninger for $d=5$ to higher dimensions. At the same time, we provide  a general framework for proving stability of self-similar blowup solutions to a large class of semilinear heat equations in arbitrary space dimension $d \geq 3$, including a robust and simple method for solving the underlying spectral problems.
\end{abstract}

\maketitle

\section{Introduction}

In this paper, we study connection 1-forms $A_{j}: \R^d \to  \mathfrak{so}(d)$, $j = 1,\dots, d$, where 
$ \mathfrak{so}(d)$ denotes the Lie algebra of the Lie group $SO(d)$, i.e., $ \mathfrak{so}(d)$ can be considered as the set
of skew-symmetric $(d\times d)$-matrices endowed with
the commutator bracket. In the following, Einstein's summation convention is in force. The associated covariant derivative acting on $\mathfrak{so}(d)$-valued functions is defined by $\mb D_j := \partial_j + [A_j,\cdot]$ and the curvature tensor amounts to  
\[ F_{j k} := \partial_{j} A_{k}- \partial_{k} A_{j} + [ A_{j}, A_{k} ]. \]
The Yang-Mills functional is then defined as
\begin{equation}\label{Eq:YangMillsFunc}
\mc F[A] = \frac{1}{2} \int_{\R^d} \mathrm{tr} \langle F_{jk}, F^{jk} \rangle dx.
\end{equation}  
The associated Euler-Lagrange equations read
\begin{align}\label{Eq:YangMills}
\mb D^{j}  F_{j k}(x)= 0
\end{align}
and solutions are referred to as Yang-Mills connections. By introducing an artificial time-dependence, the gradient flow associated to Eq.~\eqref{Eq:YangMills} yields 
\begin{equation}\label{Eq:YMGradientflow}
\partial_t A_{k}(x,t) = \mb D^{j}  F_{j k}(x,t) , \quad t > 0,
\end{equation}
for some initial condition $A_{k}(0,x) = A_{k}(x)$. This model is referred to as the \textit{Yang-Mills heat flow}
for connections on the trivial bundle $\R^{d} \times SO(d)$. The natural question concerns the existence of solutions to this initial value problem and the possibility of the formation of singularities in finite time. 

Eq.~\eqref{Eq:YMGradientflow} enjoys scale invariance, $A_{j} \mapsto A^{\lambda}_{j}$, 
\[A^{\lambda}_{j}(x,t) := \lambda A_{j}( \lambda x, \lambda^2 t), \quad \lambda > 0\]
and the model is \textit{supercritical} for $d\geq5$.  In this case it is well-known that solutions can blowup in finite time, see the works of Naito \cite{Nai94}, Grotowski \cite{Gro01} and Gastel \cite{Gas02}. The nature of singularities for the Yang-Mills heat flow over compact $d$-dimensional manifolds has been investigated by Weinkove in  \cite{Wei04} and recently by Kelleher and Streets in \cite{KelleherStreets2018}, \cite{KelleherStreets2018c}, showing that \textit{homothetically shrinking solitons} appear as blowup limits at singular points. Such objects correspond to self-similar solutions of the Yang-Mills heat flow on the trivial bundle over $\R^d$, which is the main motivation to study the problem in this geometric setting. An explicit example was given in \cite{Wei04}, namely
\begin{equation}\label{Def:EquivWeinkove}
A^{T}_{j}(x,t) =  u_T(|x|,t) \sigma_{j}(x) ,
\end{equation}
where $\sigma^{ik}_{j}(x) =    \delta^k_{j} x^i - \delta^i_{j} x^k$,
\begin{align}\label{Eq:WeinkoveSol}
u_T(r,t)  = \tfrac{1}{T-t} W\left (\tfrac{r}{\sqrt{T-t}} \right), \quad W(\rho) = \frac{1}{a \rho^2 + b}
\end{align}
for some $T >0$ and
\begin{equation}\label{Def:ConstantsWeinkove}
a =\tfrac {\sqrt{d-2}} {2 \sqrt{2}} , \quad b =  \tfrac{1}{2} (6d-12-(d+2) \sqrt{2d-4} ),
\end{equation}
which solves Eq. \eqref{Eq:YMGradientflow} in dimensions $5 \leq d \leq 9$. In this paper, we investigate the stability of this solution in the $SO(d)-$equivariant setting, i.e., we only consider connections of the form 
\begin{equation}\label{Def:Equiv}
A_{j}(x,t) = u(|x|,t) \sigma_{j}(x).
\end{equation}
It is well-known that this symmetry is preserved by the flow, see e.g. \cite{Gro01}, \cite{Gas02}. Furthermore, the system \eqref{Eq:YMGradientflow} reduces to a single equation for the function $u: [0,\infty) \times [0,\infty) \to \R$,
\begin{align} \label{Eq:EquivarEq}
\partial_t u(r,t) -   \partial^2_r  u(r,t)  - \frac{d+1}{r} \partial_r u(r,t) 
  - 3 (d-2)  u^2(r,t)  +  (d-2)  r^2 u^3(r,t) = 0,
\end{align}  
with initial condition $u(\cdot, 0) = u_0$.  The scale invariance of Eq.~\eqref{Eq:YMGradientflow} implies that
Eq.~\eqref{Eq:EquivarEq} is invariant under  $u \mapsto u_{\lambda}$,
 \begin{equation}\label{Eq:scaling}
 	u_{\lambda}(r,t)= \lambda^2 u(\lambda r,\lambda^2 t), \quad \lambda > 0. 
 \end{equation}
In particular, the explicit solution defined in \eqref{Eq:WeinkoveSol} is a self-similar blowup solution to \eqref{Eq:EquivarEq}. We note that infinitely many solutions of that type are expected to exist in dimensions $5 \leq d \leq 9$ with the profile $W$ as a ``ground state", see Biernat and Bizo{\'n} \cite{BieBiz11}. 
\subsection{The main result} \label{Sec:MainResult}
The central result of this paper is the proof that the family of self-similar blowup solutions $\{u_T:T>0\}$ is nonlinearly asymptotically stable under small perturbations of the initial data. The strategy we devise is relatively simple and straightforward but nevertheless quite general and broadly applicable.  In view of Eq.~\eqref{Eq:EquivarEq} we consider $u(|\cdot|,t)$ as a radial function on $\R^{d+2}$. 
 We define the (radial) differential operators
\begin{align}\label{Def:Radial_Operators}
D^{k} := \begin{cases} \Delta^{k/2}, & \quad \text{ for $k \in \N_0$ even}  \\ \nabla \Delta^{(k-1)/2}, & \quad \text{ for  $k \in \N_0$ odd.}  \end{cases}
\end{align}
Then we set $n := d+2$ and define
\begin{align}\label{Def:kappas}
\kappa_0 = \begin{cases} \frac{n-3}{2}, & \text{ for } n \text{ odd }  \\  \frac{n-2}{2}, & \text{ for } n \text{ even }    \end{cases}, \quad  \kappa_1 = \kappa_0 + 2.
\end{align}
Furthermore, we denote by $X$ the completion of the space of radial $C^\infty_c(\R^n)$ functions under the following norm
\begin{align}\label{Def:NormX}
\|f \|^2_X := \|  D^{\kappa_0}f \|^2_{L^2(\R^{n})} +   \|  D^{\kappa_1}f \|^2_{L^2(\R^{n})},
\end{align} 
see Section \ref{Sec:X}.  Note that the solution $u_T$ blows up both in $L^\infty$ and in the $X-$norm\footnote{Strictly speaking, since $u_T(|\cdot|,t) \notin C^\infty_c(\R^n) $ one has to show first that $u_T(|\cdot|,t)\in X$ and that its $X-$norm is indeed given by \eqref{Def:NormX}. This follows from the proof of Lemma \ref{Le:finX}.}, with the blowup rates given by
\begin{align}\label{Eq:ScalinguT}
\|u_{T}(|\cdot|,t) \|_{L^\infty(\R^n)} \simeq (T-t)^{-1} \quad \text{and} \quad
\|u_{T}(|\cdot|,t) \|_X \simeq (T-t)^{ - \frac12 (\kappa_1 + 2 - \frac{n}{2} )}.
\end{align}
In the following, we fix $T=1$ and consider the time evolution governed by \eqref{Eq:EquivarEq} for radial perturbations of the blowup initial data 
\begin{align}\label{Eq:InitialData}
u(|\cdot|,0) = u_1(|\cdot|,0) + v.
\end{align}
For the statement of the main result we use the notation
$ u(t):=u(|\cdot|,t)$.

\begin{theorem}\label{Thm:Main}
Fix $5 \leq d \leq 9$.  There exist  $\delta,M>0$ such that the following holds. For every radial $v \in C^\infty_{c}(\R^{d+2})$ satisfying  
$ \|v\|_X \leq \tfrac{\delta}{M}$
there exists $T = T(v) \in [1- \delta, 1+ \delta]$ such that the Cauchy problem given by \eqref{Eq:EquivarEq}
and \eqref{Eq:InitialData} has a unique solution $u$ for which
\begin{gather*}
	u(t) \in C^\infty(\R^{d+2}) \text{ for all } t\in [0,T),\quad \text{and} \quad
	u\in C([0,T),X)\cap C^1((0,T),X).
\end{gather*}
Furthermore, the solution blows up at $t = T$, and converges to $u_T$ in the sense that 
\begin{equation}\label{Eq:X_estimate}
\frac{\| u(t)-u_T(t) \|_X}{\|u_{T}(t) \|_X} \lesssim \delta (T-t)^{\omega},
\end{equation}
and
\begin{equation}\label{Eq:L_inf_estimate}
\frac{\| u(t)  -  u_T(t) \|_{L^{\infty}(\R^+)}}{\| u_T(t)\|_{L^{\infty}(\R^+)}} \lesssim  \delta (T-t)^{\omega},
\end{equation}
for all $t \in [0,T)$ and some $\omega > 0$.
\end{theorem}

\medskip
 The case $d=5$ has been proved by Donninger and the second author in \cite{DonSch19}. However, the approach devised there is technically involved and a generalization to more general situations is highly non-trivial. Here, we use a different strategy based on the work of Biernat, Donninger and the second author \cite{BieDon18}, \cite{BieDonSch17} on an analogous problem for the harmonic map heat flow in $d=3$, see the discussion below. We streamline, simplify and generalize this approach to all higher space dimensions (although it is applied only to $5 \leq d \leq 9$ where the blowup solution exists). One of the main contributions of this paper is the underlying spectral analysis, which is completely new in this context. For this, as opposed to the constructive approach of \cite{BieDon18}, we develop a soft, robust and general method to solve  spectral problems that underlie stable self-similar blowup in semilinear heat equations. 

\subsection{Related results}

For the Yang-Mills heat flow, notions of variational stability of shrinking solitons have been investigated by Kelleher and Streets  \cite{KelStr16} as well as by Chen and Zhang \cite{CheZha15}. However, to the best of our knowledge, our results together with \cite{DonSch19} are the only ones that prove stable blowup behavior in any sense.

Eq.~\eqref{Eq:YMGradientflow} is energy critical in $d=4$. In this case, global existence of solutions in the equivariant setting was shown by Schlatter, 
Struwe and Tahvildar-Zadeh  \cite{SchStrTah98}. For more general geometric situations, global existence was a long-standing open problem, which has been positively resolved only very recently by Waldron \cite{Wal19}. In higher space dimensions $d \geq 10$, the existence of self-similar blowup solutions to Eq.~\eqref{Eq:EquivarEq} 
was excluded by Bizo{\'n} and Wasserman \cite{BizWas15}. Instead, the generic blowup is expected to be \textit{type II}, see also the discussion on the harmonic map heat flow below. 


\subsubsection{Supercritical harmonic maps heat flow}
The Yang-Mills heat flow bears many similarities with the heat flow of harmonic maps from $\R^d$ to $\mathbb{S}^d$, which in co-rotational symmetry 
reduces to 
\begin{align}\label{Eq:HarmMaps}
\partial_t u(r,t) -   \partial^2_r  u(r,t)  - \tfrac{d-1}{r} \partial_r u(r,t) -   \tfrac{(d-1) \sin(2 u(r,t))}{2 r^2} =0 .
\end{align}
Self-similar blowup solutions exist for $3 \leq d \leq 6$, see Fan \cite{Fan99}. In \cite{BieDonSch17}, \cite{BieDon18} stable self-similar blowup has been proven in $d=3$, which confirmed prior numerical observations by Biernat and Bizo{\'n} \cite{BieBiz11}. We note that our approach can be applied to \eqref{Eq:HarmMaps} to obtain the analogous result in higher space dimensions. 

For $d \geq 7$, self-similar blowup is excluded, see \cite{BizWas15}. In this regime, type II blowup solutions have been contructed by Biernat and Seki \cite{BieSek19} and well as by Ghoul \cite{Gho17} and Ghoul, Nguyen and Tien \cite{GhoNguTie19}. It is most likely that similar results can be obtained for the Yang-Mills heat flow in dimensions $d \geq 10$. The critical case corresponds to $d=2$. In contrast to the Yang-Mills heat flow, finite-time blowup of solutions occurs, see  \cite{GroSha07} for a discussion. Stable type II blowup for $d=2$ is due to Rapha{\"e}l and Schweyer \cite{RapSch13}, \cite{RapSch14}.

\subsection{Comments on the method of proof}

The general idea is to consider the initial value problem for Eq.~\eqref{Eq:EquivarEq} in adapted coordinates in which the self-similar blowup solution becomes static and stability of finite time blowup turns into asymptotic stability of a steady state solution. For the new dynamical system, we follow the usual approach and investigate the linearized problem first. By means of semigroup theory we solve the linear problem and then treat the nonlinearity perturbatively by using fixed point arguments.

For the implementation of this strategy a suitable functional analytic setup has to be found. For the linearized problem, there is in fact a canonical choice provided by a weighted Lebesgue space $\mc H$ of radial $L^2_{\sigma}(\R^n)$ functions with $\sigma(x) = e^{-|x|^2/4}$. In this space, the underlying spectral problem is self-adjoint and defines the linear dynamics in $\mc H$. However, the weight function $\sigma$ decays at infinity, rendering this setting useless in controlling the nonlinear terms.

To get around this issue, we introduce the function space $X$ that can be thought of as an intersection Sobolev space which embeds continuously into $\mc H$. Furthermore, $X$ is invariant under the linear flow and  allows handling the nonlinearity. To control the linear flow in $X$, we exploit the decay of the underlying potential at infinity and split the problem into one on a bounded domain, where the self-adjoint growth bounds can be utilizied, and a remainder that can be made small in a suitable sense.

The analysis of the linear flow on $\mc H$ crucially relies on spectral theory for self-adjoint operators on $L^2(\R^+)$ of the following form
\begin{equation}\label{Def:Af}
	\mc Af(\rho)= -f''(\rho)+q(\rho)f(\rho).
\end{equation}
Here the potential $q$ is smooth on $\R^+$ but  singular at both endpoints, namely 
$\lim_{\rho \rightarrow 0^+}q(\rho)=\lim_{\rho \rightarrow +\infty}q(\rho)=+\infty$. In particular, one has to show that
\begin{equation}\label{Eq:spec}
	\sigma(\mc A) \subseteq \{-1\} \cup (0,\infty),
\end{equation}
where the negative eigenvalue arises as a result of time-translation symmetry. To show this, we establish a simple (integral bound) criterion for the absence of negative spectrum of radial Schr\"odinger operators of type \eqref{Def:Af}. Along with first removing $\la=-1$ from the spectrum (via a well known trick from the supersymmetric quantum mechanics) this yields a particularly simple proof of \eqref{Eq:spec}. What is more, the method we develop is general, uses little structure, and can be applied even when the underlying self-similar solution is not known in closed form.

\subsection{Notation and Conventions}
We write $\N$ for the natural numbers $\{1,2,3, \dots\}$, $\N_0 := \{0\} \cup \N$. Furthermore, $\R^+ := \{x \in \R: x >0\}$.
The notation $a\lesssim b$ means $a\leq Cb$ for an absolute constant $C>0$ and we write $a\simeq b$ if $a\lesssim b$ and $b \lesssim a$.  	
If $a \leq C_{\varepsilon} b$ for a constant $C_{\varepsilon}>0$ depending on some parameter $\varepsilon$, we write $a \lesssim_{\varepsilon} b$. 
We use the common notation $\langle x \rangle := \sqrt{1+|x|^2}$ also known as the \textit{Japanese bracket}.
For a function $x \mapsto g(x)$, we denote by $g^{(n)}(x) = \frac{d^n g(x)}{dx^n} $ the derivatives of order $n \in \N$. For $n=1,2$, we also write $g'(x)$ and $g''(x)$, respectively. 
By $C_{c,\text{rad}}^\infty(\R^n)$ we denote the space of smooth, radial functions with compact support. Also, $\mc S_{\mathrm{rad}}(\R^n)$ stands for  the space of radial Schwartz functions.
By $L^p(\Omega)$ for $\Omega \subseteq \R^n$, we denote the standard Lebesgue space. For a closed linear operator $(L, \mc D(L))$, we write $\sigma(L)$ for the spectrum.
The resolvent set is defined as $\rho(L) := \C \setminus \sigma(L)$ and we write 
$R_{L}(\lambda):=(\lambda- L)^{-1}$ for $\lambda \in \rho(L)$.

\section{Formulation of the problem in similarity coordinates}\label{Section:Formulation_SimVar}

Fix $5 \leq d \leq 9$, $d \in \N$.  We rewrite the initial value problem given by \eqref{Eq:EquivarEq} and \eqref{Eq:InitialData} in  similarity coordinates
 $(\rho,\tau) \in [0,\infty) \times [0,\infty)$ defined by 
\[ \tau = - \log(T-t) + \log T, \quad \rho = \frac{r}{\sqrt{T-t}}. \]
The blowup time $T > 0$ enters the analysis as a free parameter that will be fixed only at the very end of the argument. By setting 
\[  \psi( \tfrac{r}{\sqrt{T-t}}, - \log(T-t) + \log T)  :=  (T-t)  u(r,t), \]
we obtain
\begin{align} \label{Eq:SelfSim}
\partial_{\tau}  \psi(\rho,\tau)   =   \big ( \partial^2_{\rho}     + \tfrac{d+1}{\rho} \partial_{\rho} 
- \tfrac{1}{2} \rho \partial_{\rho}  - 1 \big ) \psi(\rho,\tau)    + 3(d-2)  \psi(\rho,\tau)^2 - (d-2) \rho^2   \psi(\rho,\tau)^3
\end{align}
with initial condition 
\begin{equation}\label{Eq:SelfSimData}
\psi(\rho,0)  = T  W \left (\sqrt{T} \rho \right )   + T v ( \sqrt{T} \rho ).
\end{equation}

The differential operator on the right hand side of Eq.~\eqref{Eq:SelfSim} has a natural extension to 
$\R^{n}$, $n = d+2$.  In fact, the evolution equation can be formulated as 
\begin{align}\label{Eq:YM_Abstract}
\frac{d}{d \tau} \Psi(\tau)  = L_0 \Psi(\tau)  + F(\Psi(\tau ))
\end{align}
with the formal differential operator
\begin{align}\label{Def:L0}
L_0 f(x) := \Delta f(x) - \frac12 x \cdot \nabla f(x) - f(x)
\end{align}
acting on radial functions and a nonlinearity defined accordingly. Note that here $\Delta$ denotes the Laplace operator on $\R^{n}$.
To study the evolution near $W$, we consider the ansatz
$\Psi(\tau) = W(|\cdot|) + \Phi(\tau)$, which
yields the central equation of this paper
\begin{align}\label{Eq:YM_AbstractPerturbation}
\begin{split}
\frac{d}{d \tau} \Phi(\tau)  & = L \Phi(\tau)  + \mc N(\Phi(\tau)) ,\quad \tau > 0 \\
\Phi(0) & =  \mc  U(v, T).
\end{split}
\end{align}
Here
\begin{align}\label{Def:L}
L := L_0 + L', \quad L'f(x) := V(|x|) f(x),
\end{align}
where the potential is given by
\begin{align}\label{Eq:Pot_explicit}
V(\rho) = \frac{3(n-4)\big (2b+(2a -1) \rho^2 \big)}{(b+ a \rho^2)^2}
\end{align}
for $a =\tfrac {\sqrt{n-4}} {2 \sqrt{2}}$, $b =  n(3 - \tfrac12 \sqrt{2n-8})-12$. The remaining nonlinear term is given by 
\begin{equation}\label{Eq:Nonlinear_term}
	\mc N(\Phi(\tau)) = (d-2) \big (3 \Phi(\tau)^2  - 3 |\cdot|^2  W(|\cdot|) \Phi(\tau)^2  -   |\cdot|^2  \Phi(\tau)^3 \big)
\end{equation}
and the initial condition transforms to
\begin{equation}\label{Def:Initial_Data}
	\mc U(v, T)  =  T v ( \sqrt{T} |\cdot| ) + T W \left ( \sqrt{T} |\cdot| \right ) -  W(|\cdot|).
\end{equation}
Now we turn to developing a functional analytic set up for the analysis of the evolution Eq.~\eqref{Eq:YM_AbstractPerturbation}.
\section{Operator analysis and semigroup theory}

\subsection{Self-adjoint theory and linear dynamics in a weighted $L^2-$space}
 
 We study first the free linear part of Eq.~\eqref{Eq:YM_AbstractPerturbation}. We define a Hilbert space $\mc H$ as a weighted $L^2$-space of radial functions
 \[\mc H := \{ f \in L^2_{\sigma}(\R^n): f \text{ is real and radial} \},\]
 with $\sigma(x) = e^{-|x|^2/4}$ and the corresponding inner product
  \[
  (f|g)_{L^2_\sigma(\R^n)}:=\int_{\R^n}f(x)g(x)\sigma(x)dx.
  \]
 Now we endow $L_0$ with a domain $\mc D(L_0):=C^\infty_{c,\text{rad}}(\R^n)$, which turns it into an unbounded linear operator on $\mc H$. The space $\mc H$ is in fact a natural choice since $L_0$ can be defined in a self-adjoint manner. Namely, there is the following result.

\begin{lemma}\label{Le:L0sa}
The closure of $L_0$ is a self-adjoint operator $\mc L_0: \mc D(\mc L_0) \subseteq \mc H \to \mc H$ which has compact resolvent and generates a strongly continuous one-parameter semigroup $\{S_0(\tau) : \tau \geq 0\}$ on $\mc H$. What is more, $S_0(\tau)$ is explicitly given by 
\begin{align}\label{Eq:FreeSemigroup}
[S_0(\tau) f](x) = e^{-\tau} (G_{\alpha(\tau)} \ast f)(e^{-\tau/2} x)
\end{align}
where $G_{\alpha(\tau)}(x) = [4 \pi \alpha(\tau)]^{-\frac{n}{2}} e^{-|x|^2/4 \alpha(\tau)}$ and $\alpha(\tau) = 1-e^{-\tau}$. 
\end{lemma}

\begin{proof}
We start with the observation that $-L_0$ is unitarily equivalent to a one-dimensional Schr\"odinger operator 
\begin{align}\label{Def:DiffExpA0}
A_0 u(\rho) = - u''(\rho) + q(\rho) u(\rho)
\end{align}
with  
\[q(\rho) := \frac{\rho^2}{16} + \frac{(n-3)(n-1)}{4 \rho^2} - \frac{n-4}{4}.\]
More precisely, $- L_0 = U A_0 U^{-1}$ for the unitary map
\[
U: L^2(\R^+) \to \mc H, \quad u \mapsto U u = |\mathbb S^{n-1}|^{-\frac12} |\cdot|^{-\frac{n-1}{2}} e^{|\cdot|^2/8} u(|\cdot|),
\] 
and $\mc D(A_0)=U^{-1} \mc D(L_0)$.
Let $A_{0,c}$ be the restriction of $A_0$ to $C_c^\infty(\R^+)$. Then by standard criteria we see that $A_{0,c}$ is limit-point at both endpoints of the interval $(0,\infty)$, see e.g. \cite{Wei87}, Theorem 6.4, p.~91 and Theorem 6.6, p.~96. 
Therefore (by say Theorem X.7 in \cite{ReeSim75}, p.~152) the closure of $A_{0,c}$  is self-adjoint, and being the only self-adjoint extension, it is necessarily given by the maximal operator $\mc A_0: \mc D(\mc A_0) \subseteq L^2(\R^+)\rightarrow L^2(\R^+) $ where
\begin{align} \label{Def:A0}
\mc D(\mc A_0) := \{ u \in L^2(\R^+): u, u' \in AC_{\mathrm{loc}}(\R^+), A_0 u \in L^2(\R^+) \},
\end{align}
and $\mc A_0 u = A_0 u$ for $u \in \mc D(\mc A_0)$.
Now, since $A_{0,c} \subseteq A_0 \subseteq \mc A_0  =  \overline{A_{0,c}}$ then $\overline{A_0}=\mc A_0$. Furthermore, the growth of $q$ at infinity implies that $\mc A_0$ has compact resolvent. 
Mapping back through $U$ we conclude that $L_0$ is essentially self-adjoint with the closure $\mc L_0 = -U\mc A_0 U^{-1}$ and $\mc D(\mc L_0)=U\mc D(\mc A_0)$. Also, since the potential $q$ is bounded from below, the operator $A_0$, its closure $\mc A_0$, and (by unitary equivalence) the operator $-\mc L_0$ are all bounded from below. The existence of a strongly continuous semigroup on $\mc H$ which is generated by $\mc L_0$ follows.
 
 Since we are simply dealing with the heat equation in rescaled variables, the formula \eqref{Eq:FreeSemigroup} is in fact derived by just transforming the usual heat semigroup. That $S_0(\tau)$ is indeed the semigroup generated by $\mc L_0$ follows from the definition of the semigroup generator. 
\end{proof}

The explicitly given semigroup $S_0(\tau)$ goes under the name Ornstein-Uhlenbeck and will be important for our analysis later on. Now we turn to the analysis of the full linear spatial operator in  Eq.~\eqref{Eq:YM_AbstractPerturbation}.

\begin{proposition}\label{Prop:Spectrum}
The closure $\mc L$ of the operator $ L:  \mc D(L):=\mc D( L_0) \subseteq \mc H  \to \mc H$ is self-adjoint, has compact resolvent, and generates a strongly continuous semigroup $\{ S(\tau) : \tau \geq 0\}$ on $\mc H$.
For the spectrum  of $\mc L$, which consists only of eigenvalues, we have
\begin{equation}\label{Eq:Spectrum_L}
	\sigma( \mathcal L) \subseteq (- \infty, 0) \cup \{1\}. 
\end{equation}
The spectral point $\lambda = 1$ is a simple eigenvalue with an eigenfunction $\mb g = g  / \| g \|_{L^2_{\sigma}(\R^n)}$, where $g(x) = (a |x|^2+b)^{-2}.$
\end{proposition}

\begin{proof}
Boundedness of the potential $V$ implies boundedness of the operator $L'$ on $\mc H$. Therefore $\mc L = \mc L_0 + L'$ and $\mc D(\mc L)=\mc D(\mc L_0)$. Furthermore, the Kato-Rellich theorem implies self-adjointness of $\mc L$,  see e.g. \cite{Tes09}, Theorem 6.4,  p.~135. Similarly, $\mc L$ has compact resolvent.  Furthermore, by the bounded perturbation theorem, the operator $\mc L$  generates a strongly continuous semigroup $\{ S(\tau) : \tau \geq 0\}$ on $\mc H$.

For the structure of the spectrum of $\mc L$, it suffices to investigate the unitarily equivalent Schr\"odinger operator $-\mc A: \mc D(\mc A) \subseteq L^2(\R^+) \to L^2(\R^+)$ defined by $ \mc D(\mc A) =  \mc D(\mc A_0)$, where $\mc A u  =  \mc A_0 u + V u$  with $\mc A_0$ from the proof of Lemma \ref{Le:L0sa}. By a straightforward calculation we see that the function
\begin{equation*}
\tilde g(\rho )= \rho^{\frac{n-1}{2}}e^{-\frac{\rho^2}{8}} (a \rho^2+b)^{-2}
\end{equation*}
belongs to $\mc D(\mc A)$ and satisfies $(A_0+V)\tilde{g}=-\tilde{g}$. Hence $-1\in\sigma(\mc A)$, i.e. $\la=1$ is an eigenvalue of $\mc L$ with an eigenfunction $\mb g$.
It remains to show that $\la=-1$ is the only non-positive spectral point of $\mc A$. Note that the fact that $\tilde{g}$ is strictly positive can be used in conjunction with a Sturm-Liouville oscillation theorem to rule out eigenvalues $\la < -1$. But this leaves the interval $(-1,0]$ untreated, and therefore different approach is necessary.

In the rest of the proof we exhibit a general method for proving spectral claims of type \eqref{Eq:Spectrum_L}. First, the fact that $\tilde g$ is strictly positive allows for a ``removal" of $\la=-1$ via a standard trick from supersymmetric quantum mechanics, see e.g. \cite{CosDonGlo17}, Sec.~3.1. More precisely, we first factorize $\mc A = A^{+}A^{-} - 1$ such that the kernel of $A^{-}$ is spanned by $\tilde g$. Then the corresponding supersymmetric expression $A^{-} A^{+} -1$  gives rise to a (maximally defined) self-adjoint operator $\mc A_S: \mc D(\mc A_S) \subseteq L^2(\R^+) \to L^2(\R^+)$ which is (modulo $\la=-1$) isospectral with $\mc A$. On $C^{\infty}_c(\R^+)$ as its core, the operator $\mc A_S$ is given by 
 \begin{equation}\label{Def:A_s}
   \mathcal{\mc A_S}u(\rho)=-u''(\rho)+\frac{n^2-1}{4\rho^2} u(\rho)+ Q(\rho)u(\rho),
 \end{equation}
with 
\begin{equation*}
	{Q}(\rho)= \frac{\rho^2}{16} - \frac{n}{4} + \frac{3}{2}-2\frac{a(2a(n-4)+b)\rho^2+b(2a(n-2)+b)}{(a\rho^2+b)^2}.
\end{equation*}
Now, it remains to show that $\mc A_S$ has no non-positive eigenvalues. To prove this we establish an integral criterion for the absence of negative spectrum for a class of operators which, in particular, contains $\mc A_\mc S$, see Appendix \ref{Sec:GGMT}.
Let $Q_-(\rho):=\min \{Q(\rho),0\}$ and
	\begin{equation}\label{Def:GGMT}
		B(n,p):=\frac{(p-1)^{p-1}\Gamma(2p)}{n^{2p-1}p^p \Gamma(p)^2}\int_{0}^{+\infty}\rho^{2p-1}|Q_-(\rho)|^p d\rho.
	\end{equation}
According to Theorem \ref{Thm:GGMT}, to rule out non-negative eigenvalues of $\mc A_S$ it is enough to prove that for every $n$ there is a choice of $p\geq1$ such that  
	\begin{equation}\label{B(n,p)}
		B(n,p)<1.
	\end{equation}  
We now fix $n=8$, and show that $B(8,4)<1$.
Since $Q(\rho)>0$ for $\rho \geq \frac{47}{10}$, we have
	\begin{equation}\label{Eq:bound}
		B(8,4)<\frac{945}{8^9}\int_{0}^{\frac{47}{10}}\rho^7Q(\rho)^4d\rho.
	\end{equation}
The integral in Eq.~\eqref{Eq:bound} can of course be easily computed numerically. However, note that the integrand is an odd rational function and has a unique partial fraction decomposition of the following form 
	\begin{equation*}
		\rho^7Q(\rho)^4=\sum_{i=0}^{6}a_i \rho^{2i+1}+\sum_{i=1}^{8}\frac{b_i \rho}{(a \rho^2+b)^i},
	\end{equation*}
for some constants $a_i, b_i$. Hence, the integral in \eqref{Eq:bound} can be explicitly computed, and this yields $B(8,4)<1$. In the same way we prove that the same holds for $B(9,4), B(10,6)$ and $B(11,6)$.
\end{proof}
\begin{remark}\label{Rem:Spectrum}
	Donninger and the second author \cite{DonSch19} proved the claim \eqref{Eq:Spectrum_L} for $d=5$ with an ad hoc method which does not carry over to higher dimensions. Also, Biernat and Donninger \cite{BieDon18} managed to prove an analogous spectral result for the harmonic map heat flow in $d=3$ by constructing a zero energy solution to the underlying spectral equation and counting the number of its oscillations. Our approach on the other hand is much softer, non-constructive, robust and likely yields a considerably simpler proof than the one in \cite{BieDon18}.

	The fact that $W$ (and hence the potential $Q$) is an explicit rational function made the above integral bound argument particularly simple. We however claim that the same approach works for arbitrary, not necessarily explicit, solutions.
	Namely, note that the condition \eqref{B(n,p)} holds for an ``open" set of potentials that are ``close" to $Q_-$. 
	Therefore, if a conjecturally stable self-similar solution is not known in closed form (say it is only observed numerically) then constructing a good enough approximation to it allows for proving, first that such solution exists and second that the estimate \eqref{B(n,p)} holds. Our procedure therefore provides a general method for proving spectral stability of (not necessarily explicit) self-similar solutions.
\end{remark}

The eigenvalue $\la=1$ is an artefact of the underlying time translation symmetry which will be controlled later on by modulating the blowup time $T$ in Eq.~\eqref{Eq:YM_AbstractPerturbation}. We therefore proceed with introducing the orthogonal projection onto the corresponding unstable mode $\mb g$,
\begin{align}\label{Def:OrthProj}
\mc P f := (f|\mb g)_{L^2_{\sigma}(\R^n)} \mb g,
\end{align}
and studying the linear flow on the complementary (stable) subspace.
More precisely, as a consequence of the definition of $\mc P$ and Proposition \ref{Prop:Spectrum} we have the following result.
\begin{corollary}\label{Cor:Decay_S}
	There exists $\omega_0>0$ such that $S(\tau) \mc P  f =  e^{\tau} \mc Pf$ and
	\begin{equation*}
		\|	S(\tau) (1- \mc P)  f \|_{{L^2_{\sigma}(\R^n)} }  \leq e^{-\omega_0 \tau} \|(1- \mc P)  f \|_{L^2_{\sigma}(\R^n)}
	\end{equation*}
	for all $f\in\mc H.$
\end{corollary}
For our further analysis, we also need growth estimates for the semigroup in the graph norm of fractional powers of $1-\mc L$. First we ensure that such operators exist.
\begin{corollary}\label{Le:Graph_L}
	There is a unique self-adjoint, positive operator $(1 - \mc L)^{\frac{1}{2}}$ with  $C^\infty_{c,\mathrm{rad}}(\R^n)$ as a core, such that $[(1 - \mc L)^{\frac{1}{2}}]^2 = 1 - \mc L$. Furthermore, the square root commutes with any bounded operator that commutes with $\mc L$. 
\end{corollary}
\begin{proof}
The existence and the basic properties of the square root are standard results. It remains to prove that $C^\infty_{c,\mathrm{rad}}(\R^n)$ is a core of $(1-\mc L)^\frac{1}{2}$. To that end, let $f \in \mc D((1 - \mc L)^{\frac{1}{2}})$ and $\varepsilon >0$ be arbitrary. 
The fact that  $\mc D(\mc L)$ is core for $(1 - \mc L)^{\frac{1}{2}}$ and $C^\infty_{c,\mathrm{rad}}(\R^n)$ is a core for $\mc D(\mc L)$ implies that there is a $\tilde f \in C^\infty_{c,\mathrm{rad}}(\R^n)$ such that $\|f - \tilde f \|_{L^2_{\sigma}(\R^n)} + \|(1- \mc L)^{\frac{1}{2}} (f - \tilde f) \|_{L^2_{\sigma}(\R^n)} < \varepsilon$ by using that $\| (1- \mc L)^{\frac{1}{2}} f\|_{L^2_{\sigma}(\R^n)} \lesssim \| (1- \mc L)f\|_{L^2_{\sigma}(\R^n)}  + \|f \|_{L^2_{\sigma}(\R^n)}$. 
\end{proof}
We now introduce the graph norms associated with a powers of $(1- \mc L)^{\frac12}$,
\begin{align}
\| f\|_{\mc G((1- \mc L)^{k/2})} := \|f \|_{L^2_{\sigma}(\R^n)} + \|(1- \mc L)^{k/2} f \|_{L^2_{\sigma}(\R^n)}
\end{align}
for $k \in \N_0, f \in \mc D((1- \mc L)^{k/2})$, and then we derive a bound on the growth of $S(\tau)$.  
\begin{proposition}\label{Prop:Growthbounds_Graph}
There exists $\omega_0 >0$ such that
\[\|S(\tau)  (1 -\mc P)  f \|_{\mc G((1- \mc L)^{k/2})}  \leq  e^{-\omega_0 \tau} \|(1 - \mc P) f \|_{\mc G((1- \mc L)^{k/2})} \]
for all $k \in \N_0$, $f \in \mc D((1- \mc L)^{k/2})$ and $\tau \geq 0$.
\end{proposition}

\begin{proof}
The operators $\mc P$, $\mc L$ and $S(\tau)$ mutually commute. This implies that $(1-\mc L)^{\frac12}$ commutes with the projection and the semigroup and thus, the same holds for $(1-\mc L)^{\frac{k}{2}}$, $k \in \N$. From this and Corollary \ref{Cor:Decay_S} it follows that for all $f \in \mc D((1- \mc L)^{k/2})$,
\begin{align*}
\|(1- \mc L)^{k/2}  S(\tau)   (1- \mc P) f \|_{L^2_{\sigma}(\R^n)}  &   =  \|S(\tau) (1- \mc P) (1- \mc L)^{k/2} f \|_{L^2_{\sigma}(\R^n)} \\
&  \leq e^{-\omega_0 \tau} \|(1- \mc P) (1- \mc L)^{k/2}   f \|_{L^2_{\sigma}(\R^n)} \\
&   =  e^{-\omega_0 \tau} \|(1- \mc L)^{k/2} (1- \mc P)f \|_{L^2_{\sigma}(\R^n)},
\end{align*}
which implies the claim.
\end{proof}
The space $\mc H$ is not closed under multiplication, which makes it inconvenient for the analysis of the full (nonlinear) Eq.~\eqref{Eq:YM_AbstractPerturbation}. However,  the space $X$ introduced in Section \ref{Sec:MainResult} allows for such analysis. In the following section we embed $X$ continuously in $\mc H$ and prove a number of its features that will be important later on. 

\subsection{Properties of the space $X$} \label{Sec:X}

We first recall definitions \eqref{Def:Radial_Operators} and \eqref{Def:kappas} and then endow $C^\infty_{c,\mathrm{rad}}(\R^n)$ with an inner product
\begin{equation*}
	(f|g)_X := (  D^{\kappa_0}  f | D^{\kappa_0}  g)_{L^2(\R^n)}	+ (  D^{\kappa_1}  f | D^{\kappa_1}  g)_{L^2(\R^n)}.
\end{equation*}
Note that the norm induced by this inner product is given in \eqref{Def:NormX} and therefore the completion of the pre-Hilbert space $(C^\infty_{c,\mathrm{rad}}(\R^n),(\cdot|\cdot)_X)$ is precisely the space $X$. We proceed with proving a crucial embedding property of $X$.
\begin{lemma}\label{Lem:EmbXH}
The space $X$ can be continuously embedded into $\mc H$. 
\end{lemma}
\begin{proof}
	For $f\in C^\infty_{c,\mathrm{rad}}(\R^n)$ we have
	\begin{align}\label{Eq:L_inf}
	\| f \|_{L^{\infty}(\R^n)}  \lesssim & \| \mc F f\|_{L^1(\B^n)} + \| \mc F f\|_{L^1(\R^n \setminus \B^n)} \nonumber  \\
	 \lesssim &\| |\cdot|^{-\kappa_0} \|_{L^2 (\B^n)}   \| |\cdot|^{\kappa_0} \mc F f\|_{L^2 (\B^n)}    + \| |\cdot|^{-\kappa_1} \|_{L^2(\R^n \setminus \B^n)}  \| |\cdot|^{\kappa_1} \mc F f\|_{L^2(\R^n \setminus \B^n)} \nonumber \\
	\lesssim & \| D^{\kappa_0}f \|_{L^2(\R^n)}+\| D^{\kappa_1}f \|_{L^2(\R^n)}\lesssim
	\| f \|_X. 
	\end{align}
In view of the exponential decay of the weight function, we immediately obtain that 
\begin{equation}\label{Eq:XembH}
	\| f \|_{L^2_{\sigma}(\R^n)} \lesssim \|f \|_{L^{\infty}(\R^n)} \lesssim \|f \|_X
\end{equation}
for all $f  \in C^\infty_{c,\mathrm{rad}}(\R^n)$.  Now, let $f \in X$. Then there is a sequence $(f_j)_{j \in \N} \subseteq  C^\infty_{c,\mathrm{rad}}(\R^n)$ such that $f_j \to f$ in $X$.
By the above inequality $(f_j)_{ j\in \N}$ is a Cauchy sequence in $\mc H$ and we denote its limit by $g$. We define $\iota: X \to \mc H$ by $\iota(f) := g$ and show that it is injective. In fact, if $\iota(f) = 0$, there is a sequence  $(f_j)_{j \in \N} \subseteq  C^\infty_{c,\mathrm{rad}}(\R^n)$ such that $f_j \to f$ in $X$ and $f_j \to 0$ in $\mc H$. Assume that $f \neq 0$. 
Then $(D^{\kappa_0} f_j)_{j \in \N}$ and $(D^{\kappa_1} f_j)_{j \in \N}$ are Cauchy sequences in $L^2(\R^d)$ converging to some $\tilde f_0$, $\tilde f_1$ strongly and thus also in the sense of distributions. The assumption on $f$ implies that at least one of two limit functions has to be nonzero. However, for every test function $\varphi$ and every $k \in \N$ we obtain that 
\[ |( D^{k} f_j|\varphi)_{L^2(\R^n)}| = |( f_j| D^{k} \varphi)_{L^2(\R^n)}| \lesssim \|f_j \|_{L^2_{\sigma}(\R^n)} \to 0. \]
By uniqueness of distributional limits we have a contradiction, which shows that $f=0$. The continuity of the embedding now follows from Eq.~ \eqref{Eq:XembH}.
\end{proof}
The identification of every member of $X$ with a  function in $\mc H$ allows for taking products of elements from $X$. What is more  we have the following result.
\begin{lemma}\label{Lem:Banach_alg}
	The space $X$ is closed under multiplication. Furthermore,
	\begin{align}\label{Eq:multip}
	\| f g\|_{X} \lesssim \|f\|_X \|g \|_X
	\end{align}
	for all $f,g \in X$. 
\end{lemma}
\begin{proof}
	It is enough to prove Eq.~\eqref{Eq:multip} for $f,g \in C^\infty_{c,\mathrm{rad}}(\R^n)$, as the full claim follows by density and Eq.~\eqref{Eq:XembH}. We proceed as usual and use that  
	\begin{align}\label{Eq:Leibnitz}
	D^{k}(fg) = f D^{k} g + g D^k f + \sum_{j=1}^{k-1} c_j D^{j} f D^{k-j} g  + \sum_{j_0+j_1+j_2 = k:  j_i \geq 1} \tilde c_{j} |\cdot |^{-j_0} D^{j_1} f D^{j_2} g
	\end{align} 
	for some constants $c_j >0$ and $\tilde c_j \in \Z$. For $k = \kappa_0$, the control of the first two terms follows from \eqref{Eq:L_inf}. By H\"older and Gagliardo-Nirenberg inequalities we have 
	\begin{align*}
	\|D^{j}  f D^{k-j} g& \|_{L^2(\R^{n})}   \leq \| D^{j} f \|_{L^{\frac{2k}{j}}(\R^n)}  \|D^{k-j} g \|_{L^{\frac{2k}{k-j}}(\R^n)}\\
	 & \lesssim   \left (\|g \|_{L^{\infty}(\R^n)}  \|D^k f \|_{L^2(\R^n)}  \right)^{j/k} \left ( \|f\|_{L^{\infty}(\R^n)}  \|D^k g \|_{L^2(\R^n)} \right )^{1-j/k}  \lesssim \|f\|_X \|g \|_X,
	\end{align*}
	and for $k = \kappa_1$ the argument is similar. The second sum is treated similarly along with Hardy's inequality.
\end{proof}
\begin{remark}
	Lemmas \ref{Lem:EmbXH} and \ref{Lem:Banach_alg} contain crucial properties of $X$ which are necessary for the analysis of the nonlinear flow. The key for these results is estimate \eqref{Eq:L_inf} which restricts the choice of $\kappa_0$ and $\kappa_1$ to the ones for which $\kappa_0<n/2<\kappa_1$.
\end{remark}
The following result is important and relies on the strong decay of the exponential weight. 
\begin{lemma}\label{Le:XhookrightH}
Let  $k \in \{0,\dots,  \kappa_1\}$. Then $X \subseteq \mc D((1- \mc L)^{k/2})$ and 
\begin{align}\label{Ineq:Emb_Graph}
\|(1- \mc L)^{k/2} f \|_{L^2_{\sigma}(\R^n)}  \lesssim  \|f \|_X 
\end{align}
for all $f \in X$.
\end{lemma}

\begin{proof}
We prove  Eq.~\eqref{Ineq:Emb_Graph} for $f \in C^\infty_{c,\mathrm{rad}}(\R^n)$ only, as the full claim follows from density, Lemma \ref{Lem:EmbXH} and the closedness of $(1- \mc L)^{k/2}$. First, we show that for polynomially bounded functions  $w \in C^{\infty}(\R^n)$, 
\begin{align}
 \| w D^{k} f \|_{L^2_{\sigma}(\R^n)} \lesssim  \|f \|_X,
\end{align}
for   $k \in \{0,\dots,  \kappa_1\}$.
 In fact, by exploiting the decay of the exponential weight and Hardy's inequality
 \begin{equation}\label{Eq:Hardy}
 	\||\cdot|^{-j}f\|_{L^2(\R^n)} \lesssim \| D^jf \|_{L^2(\R^n)}, \quad 0 \leq j < \frac{n}{2},
 \end{equation}
  see e.g. \cite{MusSch13}, Theorem $9.5$, p.~243, we get for  $k \in \{ 0,\dots,  \kappa_0\}$, 
\begin{align}\label{Ineq:Emb_Dk1}
 \|w D^{k} f \|_{L^2_{\sigma}(\R^n)} \lesssim  \| |\cdot|^{-\kappa_0 + k} D^k f\|_{L^2(\R^n)} \lesssim  \| D^{\kappa_0} f\|_{L^2(\R^n)}
 \end{align}
since $0\leq \kappa_0-k < \frac{n}{2}$. By the same token,  for $k \in \{\kappa_1 - 1, \kappa_1 \}$,
\begin{align*}\label{Ineq:Emb_Dk2}
 \|w D^{k} f \|_{L^2_{\sigma}(\R^n)} \lesssim  \| |\cdot|^{-\kappa_1 + k} D^k f\|_{L^2(\R^n)} \lesssim  \| D^{\kappa_1} f\|_{L^2(\R^n)}.
  \end{align*}
Now, it is easy to see that for $f \in  C^\infty_{c,\mathrm{rad}}(\R^n)$ and certain smooth polynomially bounded functions $w_j$ we have
\[ \|(1- \mc L)^{k/2} f \|_{L^2_{\sigma}(\R^n)}  \lesssim \| D^{k} f \|_{L^2_{\sigma}(\R^n)}  + \sum_{j = 0}^{k-1} \| w_j D^{j} f \|_{L^2_{\sigma}(\R^n)}, \]
which in turn implies Eq.~\eqref{Ineq:Emb_Graph}
for $k \in \{0 ,\dots,  \kappa_1\}$. 
\end{proof}

We also need the following simple observation.

\begin{lemma}\label{Le:Fact_L}
   For $f \in C^\infty_{c,\mathrm{rad}}(\R^n)$ we have
	\begin{align*}
	\| (1 - \mc L)^{\frac{1}{2}}f \|_{L^2_{\sigma}(\R^n)} = \|B f\|_{L^2_{\sigma}(\R^n)},
	\end{align*}
	 where 
	\[B f(x) = \mb g (x) \frac{d}{d|x|}  [ \mb g(x)^{-1} f(x) ] .\]
\end{lemma}
\begin{proof}
	On $C^\infty_{c,\mathrm{rad}}(\R^n)$ we can write $1 - \mc  L  = B^* B$, where $ B$ is defined above and
	\begin{align*}
	B^* f(x) =  - \frac{ 1}{\mb g(x)\mu(x)} \frac{d}{d|x|}  [  \mu(x)\mb g (x) f(x) ],
	\end{align*}
	with $\mu(x) := |x|^{n-1} e^{-|x|^2/4}$. It is easy to check that $B$ and $B^*$ are formally adjoint to each other. This implies that for all $f \in C^\infty_{c,\mathrm{rad}}(\R^n)$,
	\begin{align*}
	\|(1 - \mc L)^{\frac{1}{2}} f \|^2_{L^2_{\sigma}(\R^n)}  = ((1 - \mc L)f | f)_{L^2_{\sigma}(\R^n)} = (B^* B f|f)_{L^2_{\sigma}(\R^n)} = \|B f\|^2_{L^2_{\sigma}(\R^n)}.
	\end{align*}
\end{proof}
Recall that for $f \in C^\infty_{c,\mathrm{rad}}(\R^n)$ we have 
$ \mc L=\Delta  - \Lambda  + V(|\cdot|),
$
where 
\[\Lambda f(x) := \frac{1}{2} x \cdot \nabla f(x) +  f(x).\]
In the sequel we will estimate the elements of $ \Im \mc L $ in local Sobolev norms and for that we need to understand the interaction between $\mc L$ and the radial operators $D^k$. In particular, we have the commutator relation
\begin{align}\label{Eq:Commutator}
D^{k} \Lambda  = \Lambda D^{k} + \frac{k}{2} D^k,
\end{align}
which is  necessary for Lemma \ref{Le:DissEst} and is used in the following result. For simplicity, from now on we denote both the potential and its radial representative by just $V$.
\begin{lemma}\label{Le:Control_boundednorms}
	Let $k \in \N_0$ and $R > 0$. Then there is a constant $C_{R,k}  > 0$ such that
	\[ \|D^k f \|_{L^{2}(\mathbb B_R^n)} \leq  C_{R,k} \sum_{j=0}^{k} \|f \|_{\mc G((1- \mc L)^{j/2})} \]
	for all $f \in C^\infty_{c,\mathrm{rad}}(\R^n)$.
\end{lemma}

\begin{proof}

	We use induction. For $k=1$ there is the following estimate
	\begin{align*}
	\|D f \|_{L^{2}(\mathbb B_R^n)} & \simeq_R \| |\cdot|^{\frac{n-1}{2}} f'(|\cdot|) \|_{L^2(0,R)} \lesssim_R \| |\cdot|^{\frac{n-1}{2}} e^{-|\cdot|^2/8}  f'(|\cdot|) \|_{L^2(0,\infty)} \\
	&  \lesssim_R \|B f \|_{L^2_{\sigma}(\R^n)} + \|f\|_{L^2_{\sigma}(\R^n)}  = \|(1-\mc L)^{1/2} f \|_{L^2_{\sigma}(\R^n)} + \|f\|_{L^2_{\sigma}(\R^n)}.
	\end{align*}
	We assume that the statement holds up to some $k \in \N$. 
	Note that 
	\begin{align*}
	\|D^{k+1}    f \|^2 _{L^{2}(\mathbb B_R^n)}  &  \lesssim \|D^{k-1} (1 - \mc L) f \|^2_{L^{2}(\mathbb B_R^n)} +  \| D^{k-1} ( \Lambda - V + 1)  f\|_{L^{2}(\mathbb B_R^n)}^2 \\
	& \lesssim 
	\|D^{k-1} (1 - \mc L) f \|^2_{L^{2}(\mathbb B_R^n)}  + \| (\Lambda + \tfrac{k-1}{2})  D^{k-1} f\|_{L^{2}(\mathbb B_R^n)}^2  + \|D^{k-1} (Vf) \|_{L^{2}(\mathbb B_R^n)}^2  \\ &  \lesssim_{R,k}
	\|D^{k-1} (1 - \mc L) f \|^2_{L^{2}(\mathbb B_R^n)} +  \sum_{j=0}^{k} \|D^{j} f\|_{L^{2}(\mathbb B_R^n)}^2,
	\end{align*}
	and by assumption we have
	\begin{align*}
	\|D^{k-1}  (1 - \mc L) f \|_{L^{2}(\mathbb B_R^n)}  \lesssim_{R,k}  \sum_{j=0}^{k-1} \|(1-\mc L) f \|_{\mc G((1- \mc L)^{j/2})} 
	\lesssim_{R,k} \sum_{j=0}^{k+1} \|f \|_{\mc G((1- \mc L)^{j/2})}  
	\end{align*}
	and $$ \sum_{j=0}^{k} \|D^{j} f\|_{L^{2}(\mathbb B_R^n)} \lesssim \sum_{j=0}^{k} \|f \|_{\mc G((1- \mc L)^{j/2})}.$$
	The statement for $k+1$ follows.
\end{proof}

We conclude this section with a simple but important observation.
\begin{lemma}\label{Le:finX}
	Let $f \in C^{\infty}_\mathrm{rad}(\R^n)$ and assume that 
	\[ |D^{k} f(x)| \lesssim \langle x \rangle^{-2-k} \]
	for all $x \in \R^{n}$ and all $k \in \{0,\dots,\kappa_1\}$. Then $f \in X$. What is more, $f \in \mc D(\mc L)$ and $\mc Lf  \in X.$
\end{lemma}
\begin{proof}
	 The first claim follows from a straightforward approximation argument. Namely, given smooth and radial $\varphi$ such that $\varphi(x)=1$ for $ |x| \leq 1 $ and  $\varphi(x)=0$ for $ |x| \geq 2 $, the sequence $f_n:= \varphi(\cdot/n)f$ is Cauchy in $X$  and converges to $f$ in $L^\infty(\R^n)$ (and hence in $\mc H$), therefore $f \in X$.
	  For the second claim, note that $\mc L$ acts as a classical differential operator on smooth functions in $\mc D(\mc L)$ (this follows say from the unitary equivalence of $\mc L$ and $-\mc A$, and the fact that $\mc A$ is a differential operator on its domain, see the proof of Proposition \ref{Prop:Spectrum}). Furthermore, the function $\mc Lf$ satisfies the assumptions of the Lemma and hence  $\mc Lf\in X$.
\end{proof}

With all these results at hand, we can now turn to the investigation of the linear flow on $X$.

\subsection{The linear time evolution on $X$}  
By using Lemma \ref{Lem:Banach_alg} and the explicit form of the semigroup $S_0(\tau)$, we obtain the following Proposition.

\begin{proposition}
The restriction of $\{S(\tau): \tau \geq 0\}$ to $X$ defines a strongly continuous one-parameter semigroup $\{S_X(\tau): \tau \geq 0\}$ on $X$.  Its generator is given by the part of $\mc L$ in $X$, 
\begin{align*}
\mc L_X f  := \mc L f, \quad \mc D(\mc L_X) := \{ f \in \mc D(\mc L) \cap X: \mc Lf \in X \}.
\end{align*}
Furthermore, $\mc S_{\mathrm{rad}}(\R^n)$ is a core of $\mc L_X$.
\end{proposition}

\begin{proof}
We first exploit the explicit representation \eqref{Eq:FreeSemigroup} to prove the invariance of $X$ under $S_0(\tau)$, and perturbatively extend this to $S(\tau)$. Since $G_{\alpha(\tau)} \in L^1(\R^n)$ for every $\tau >0$, from Eq.~\eqref{Eq:FreeSemigroup} via Young's inequality we get
\begin{align}\label{Est:FreeEvol_Growth}
\|D^{k} S_0(\tau) f\|_{L^2(\R^n)} \lesssim e^{ \frac{1}{2} (\frac{n}{2} - 2 - k )\tau}\| D^k f \|_{L^2(\R^n)}
\end{align}
for all $k \in \N_0$. In particular, $X$ is invariant under $S_0(\tau)$ for all $\tau \geq 0$. By using again the explicit form of the semigroup, rescaling, and Minkowski's inequality we get
\begin{align*}
\|D^k[ S_0(\tau) f  & - f] \|_{L^2(\R^n)}  \\
& \lesssim e^{-\frac{k}{2} \tau} \int_{\R^n} \| (D^k f)(e^{-\frac{\tau}{2}}(\cdot) - \alpha(\tau)^{\frac12} y) - D^k f \|_{L^2(\R^n)} dy.
\end{align*}
Furthermore, by dominated convergence we infer that  $\|D^k[ S_0(\tau) f - f] \|_{L^2(\R^n)} \to 0$ as $\tau \to 0^+$. This shows that the free semigroup is strongly continuous on $X$. By Lemma \ref{Lem:EmbXH} and a standard result from semigroup theory, see \cite{EngNag00}, p.~60, Chap.~2, Sec.~2.3, we infer that the part of the operator $\mc L_0$ in $X$ defined as $\mc L_0|_X f := \mc L_0 f$, 
\[\mc D(\mc L_0|_X) =  \{ f \in \mc D(\mc L_0) \cap X: \mc L_0 f \in X \}\]
generates the restricted semigroup $\{S_0|_X(\tau): \tau \geq 0 \}$. An application of Lemma \ref{Le:finX} shows that $V \in X$, and thus
\begin{align}\label{eq:BoundPot}
\| V f\|_{X} \lesssim \|V \|_{X} \|f \|_X 
\end{align}
by Lemma \ref{Lem:Banach_alg}. Now, by the bounded perturbation theorem, $\mc L_X$  generates a strongly continuous semigroup $\{S_X(\tau): \tau \geq 0 \}$ on $X$. 
It remains to prove that $S_X(\tau)$ is in fact a restriction of $S(\tau)$ on $X$. For $f\in X$, $\{S(\tau)f:\tau\geq 0\}$ and $\{S_X(\tau)f:\tau\geq 0\}$ are continuous curves in $\mc H$ and $X$ respectively, and, due to boundedness of $L'$ on both $\mc H$ and $X$, both curves satisfy the following integral equation in $u$
\begin{equation}\label{Eq:Duhamel}
	u(\tau)=S_0(\tau)f-\int_0^\tau S_0(\tau-s)L'u(s)ds,
\end{equation} 
see \cite{Kat95}, Chap.~9, Sec.~2.1. But, by Gronwall's lemma, Eq.~\eqref{Eq:Duhamel} has a unique solution in $C([0,\infty),\mc H)$. Therefore $S(\tau)$ and $S_X(\tau)$ agree on $X$ and the first claim of the proposition follows. The density of $\mc S_{\mathrm{rad}}(\R^n)$ in  $X$  and the fact that $S_0(\tau)$ leaves radial Schwartz functions invariant implies that $\mc S_{\mathrm{rad}}(\R^n)$ is a core of $\mc L_0|_X$ and hence of $\mc L_X$ as a bounded perturbation.
\end{proof}

\begin{remark}
	From Eq.~\eqref{Est:FreeEvol_Growth} it follows that the free time evolution is growing exponentially in homogeneous Sobolev norms below scaling, hence having to choose both $\kappa_0$ and $\kappa_1$ greater than $n/2-2$.
\end{remark}

\begin{lemma}
The projection operator  $\mc P$ defined in Eq.~\eqref{Def:OrthProj} induces a (non-orthogonal) projection $\mc P_X$ on $ X$,
\[ \mc P_X f = (f|\mb g)_{L^2_{\sigma}(\R^n)} \mb g \quad \text{ for } f \in X, \] 
which commutes with operator $\mc L_X$ and the semigroup $S_X(\tau)$ for all $\tau  \geq 0$. Furthermore, 
\[\mathrm{ker} ~ \mc P_X  = \{ f \in X:  (f|\mb g)_{L^2_{\sigma}(\R^n)}  = 0 \}.\] 
\end{lemma}

\begin{proof}
The decay of $\mb g$ and Lemma \ref{Le:finX} imply that $\mb g \in X$. By Cauchy-Schwarz and relation \eqref{Eq:XembH} we get
\[ \| \mc P_X f \|_{X} = \big  |(f|\mb g)_{L^2_{\sigma}(\R^n)} \big | \|\mb g\|_X  \leq \|f \|_{L^2_{\sigma}(\R^n)}  \|\mb g\|_X  \lesssim \|f \|_X.  \]
The other properties follow from the properties of the semigroup on $\mc H$. 
\end{proof}

In the following, when there is no confusion, we drop the subscript for $\mc L_X$, $S_{X}(\tau)$  and  $\mc P_X$ for the sake of readability.
To derive a suitable growth bound for the semigroup on $X$, the following Lemma is crucial. 

\begin{lemma}\label{Le:DissEst}
Let $R \geq 1$. Then for all $f \in  \mc D(\mc L_X)$ we have 
\begin{align}\label{Inequ:DissEst}
(\mc L f| f)_X \leq (- \tilde{\omega} + \tfrac{C}{R^2} )\|f \|^2_X + C_R \sum_{j=0}^{\kappa_1} \|f \|^2_{\mc G((1- \mc L)^{j/2})}
\end{align}
with $\tilde{\omega}  : = \frac{1}{2}(\kappa_0 - \frac{n}{2} +2) >0$.
\end{lemma}

\begin{proof}
We prove Eq.~\eqref{Inequ:DissEst} for $f \in \mc S_{\mathrm{rad}}(\R^n)$ only, as the full claim follows from Lemma \ref{Le:XhookrightH}, the closedness of $\mc L_X$ and the fact that $\mc S_{\mathrm{rad}}(\R^n)$ is its core. The commutator relation \eqref{Eq:Commutator} yields 
\[ D^{k} \mc L f = D^{k +2} f -  D^{k} \Lambda f + D^k (V f)  =  D^{k +2} f  -  \Lambda D^k f - \frac{k}{2} D^k f +  D^k (V f). \]
We then use $(\Lambda f| f)_{L^2(\R^n)} = (1 - \frac{n}{4}  ) \| f \|^2_{L^2(\R^n)}$ to obtain the following estimate
\begin{align*}
\begin{split}
(  D^{k} \mc L f | D^{k}  f)_{L^2(\R^n)} &  \leq - (\Lambda D^{k}  f|D^{k}  f)_{L^2(\R^n)} -\tfrac{k}{2} \| D^k  f \|^2_{L^2(\R^n)}   + (D^{k} (V f) |  D^{k}  f )_{L^2(\R^n)} \\
& \leq -\tfrac12 (k - \tfrac{n}{2} + 2)  \| D^{k}  f \|^2_{L^2(\R^n)}  + (D^k (V f) |  D^{k}  f )_{L^2(\R^n)}.
\end{split}
\end{align*}
This in turn yields the bound
\begin{align*}
 (\mc L f| f)_X \leq -\tilde{\omega} \|f \|^2_X  + (V f|f)_{X}.
\end{align*}
To estimate the last term, we exploit the decay of the potential at infinity and Lemma \ref{Le:Control_boundednorms}. We use the Leibnitz formula \eqref{Eq:Leibnitz} and estimate for $j = 0, \dots, k$,
\begin{align*}
|( D^{j} V D^{k-j} f| D^{k} f)_{L^2(\R^n)}| \leq \||D^{j}   V|^{\frac{j+1}{j+2}} D^{k-j} f \|^2_{L^2(\R^n)}  +  \||D^{j} V|^{\frac{1}{j+2}} D^{k} f \|^2_{L^2(\R^n)}.
\end{align*}
For the last term, we get for every $R \geq 1$,
\begin{align*}
\||D^{j} V|^{\frac{1}{j+2}} D^{k} f \|_{L^2(\R^n)} &  \leq C_R \|D^{k} f \|_{L^2(\B^n_R)} + \tfrac{C}{R} \|D^{k} f \|_{L^2(\R^n\setminus \B^n_R)} \\
& \leq  C_R \sum_{i = 0}^{k}  \|f \|_{\mc G((1- \mc L)^{i/2})}   +  \tfrac{C}{R} \|D^{k} f \|_{L^2(\R^n)}.
\end{align*} 
For the first term we argue similarly and use Eq.~\eqref{Ineq:Emb_Dk1} to obtain
\begin{align*}
\||D^{j} V|^{\frac{j+1}{j+2}} D^{k-j} f \|_{L^2(\R^n)}   \leq C_R \sum_{i = 0}^{k-j}  \|f \|_{\mc G((1- \mc L)^{i/2})}   +   C\| | \cdot |^{-j} D^{k-j} f \|_{L^2(\R^n\setminus \B^n_R)}.
\end{align*}
For $k = \kappa_0 < \frac{n}{2}$, Hardy's inequality yields 
\[ |(D^{\kappa_0} (V f) |  D^{\kappa_0}  f )_{L^2(\R^n)}| \leq C_R \sum_{j = 0}^{\kappa_0}  \|f \|^2_{\mc G((1- \mc L)^{j/2})}   +  \tfrac{C}{R^2} \|D^{\kappa_0} f \|^2_{L^2(\R^n)}. \]
For $k = \kappa_1 > \frac{n}{2}$, we estimate
\[   \| |\cdot|^{-j} D^{\kappa_1-j} f \|_{L^2(\R^n)}  \lesssim \| D^{\kappa_1} f \| _{L^2(\R^n)} \]
for $j = 0,\dots \kappa_0$ and treat separately the cases $j \in  \{ \kappa_1 - 1, \kappa_1 \}$ for which we get 
\[
  \| |\cdot|^{-j} D^{\kappa_1-j} f\|_{L^2(\R^n\setminus \B^n_R)} \lesssim
  \| |\cdot|^{-j+2} D^{\kappa_1-j} f \|_{L^2(\R^n\setminus \B^n_R)}
    \lesssim  \| D^{\kappa_0} f \| _{L^2(\R^n)}.
\]
The second sum is treated similarly along with Hardy's inequality. This implies that 
\[ |(V f|f)_{X}| \leq  C_R \sum_{j = 0}^{\kappa_1}  \|f \|^2_{\mc G((1- \mc L)^{j/2})} +   \tfrac{C}{R^2 } \| f \|^2_X \]
for some constants $C, C_R >0$ and Eq.~\eqref{Inequ:DissEst} follows.
\end{proof}

Finally, we obtain the desired growth bound for the linearized time evolution on $X$. 

\begin{proposition}\label{Prop:DecayX}
There exists $\omega > 0$ such that
\begin{align}\label{Eq:Decay_on_X}
\| S(\tau) (1- \mc P) f \|_X   \lesssim e^{-\omega \tau } \|(1- \mc P) f \|_X,
\end{align}
for all $f \in X$ and all $\tau \geq 0$. 
\end{proposition}

\begin{proof}
We prove Eq.~\eqref{Eq:Decay_on_X} only for $f \in C^\infty_{c,\mathrm{rad}}(\R^n)$, as the full claim follows from density and the boundedness of operators $\mc P$ and $\mc S(\tau)$. From Lemma \ref{Le:finX} it follows that $\tilde{f}:=(1-\mc P)f$ belongs to $\mc D(\mc L_X)$. Therefore, we can use Lemma \ref{Le:DissEst}, Proposition \ref{Prop:Growthbounds_Graph} and Lemma \ref{Le:XhookrightH} and choose $R >0$ sufficiently large to obtain 
\begin{align*}
\tfrac{1}{2}  \tfrac{d}{d\tau} \|S(\tau)\tilde{f}  \|^2_{X} &   = ( \partial_{\tau} S(\tau) \tilde{f}|S(\tau) \tilde{f})_X  = ( \mc L S(\tau)  \tilde{f}|S(\tau) \tilde{f})_X  \\
& \leq (-\tilde{\omega}   + \tfrac{C}{R^2}) \|S(\tau) \tilde{f}\|^2_X   + 
C_R \sum_{j=0}^{\kappa_1} \|S(\tau)\tilde{f} \|^2_{\mc G((1- \mc L)^{j/2})}  \\
& \leq  -\tfrac{\tilde{\omega} }{2} \|S(\tau) \tilde{f}\|^2_X+ C e^{-2 \omega_0 \tau}  \sum_{j=0}^{\kappa_1} \| \tilde{f} \|^2_{\mc G((1- \mc L)^{j/2})} \\
&  \leq  -\tfrac{\tilde{\omega} }{2}  \|S(\tau) \tilde{f}\|^2_X +  C e^{-2 \omega_0 \tau}  \|\tilde{f} \|^2_X  \leq - 2 c_0  \|S(\tau) \tilde{f}\|^2_X  + C e^{-4 c_0 \tau}  \|\tilde{f} \|^2_X
\end{align*}
for $c_0 = \frac{1}{2} \min \{\omega_0, \tfrac{\tilde{\omega} }{2} \}$. This inequality can be written as 
\begin{align*}
\tfrac{1}{2} \tfrac{d}{d\tau}  \left[ e^{4 c_0 \tau}   \|S(\tau) \tilde{f}\|^2_X  \right] \leq C \|\tilde{f} \|^2_X
\end{align*}
and integration yields
\[   \|S(\tau) \tilde{f}\|^2_X \leq (1 + 2 C \tau ) e^{-4 c_0 \tau}  \|\tilde{f} \|^2_X  \lesssim e^{-2 \omega \tau}   \|\tilde{f} \|^2_X ,\]
for some suitably chosen $\omega  > 0$. This completes the proof.
\end{proof}

\section{Nonlnear time evolution}

\subsection{Estimates for the nonlinearity}
In this section we prove a contraction property of the nonlinearity $\mc N$, see Eq.~\eqref{Eq:Nonlinear_term}, which is necessary in order to run a fixed point argument. First, we need one more property of $X$.

\begin{lemma}\label{Lem:Cubic}
	We have
	\begin{align}\label{Est:Cubic}
	\||\cdot|^2 f_1 f_2  f_3\|_X  \lesssim  \prod_{j=1}^{3} \| f_j \|_X,
	\end{align}
	for all $f_1,f_2,f_3 \in X$.
\end{lemma}
\begin{proof}
	We prove Eq.~\eqref{Est:Cubic} for functions in $C^\infty_{c,\mathrm{rad}}(\R^n)$ only, since the full claim follows by density and Eq.~\eqref{Eq:XembH}. To establish Eq.~\eqref{Est:Cubic}, we have to control $D^k(|\cdot|^2f_1f_2f_3)$ in  $L^2(\R^n)$ for $k\in\{\kappa_0,\kappa_1\}$. To do this we treat small and large values of the variable separately. Namely, for $f \in C^\infty_{c,\mathrm{rad}}(\R^n)$ we write
	\begin{align}\label{Eq:Integralsplit}
	\|D^{k}( |\cdot|^2 f)\|_{L^2(\R^n)} = \|D^{k}( |\cdot|^2  f)\|_{L^2(\B^n)}   + \|D^{k}( |\cdot|^2  f)\|_{L^2(\R^n\setminus \B^n)}. 
	\end{align}
	By the Leibnitz rule \eqref{Eq:Leibnitz} and Hardy's inequality \eqref{Eq:Hardy} we get
	\begin{align*}
	\|D^{k}( |\cdot|^2 f)\|_{L^2(\B^n)}  \lesssim  \sum_{r=0}^{2}   \| |\cdot|^{2-r} D^{k-r} f \|_{L^2(\B^n)}  \lesssim  \sum_{r=0}^{2}   \| |\cdot|^{-r} D^{k-r} f \|_{L^2(\B^n)}     \lesssim \|f \|_{X} .
	\end{align*}
	Using this and the property \eqref{Eq:multip} we get
	\begin{align*}
	\| |\cdot|^2  f_1 f_2  f_3 \|_X  \lesssim   \prod_{j=1}^{3} \| f_j \|_X    +  \|D^{\kappa_0}(|\cdot|^2  f_1 f_2  f_3) \|_{L^2(\R^n\setminus \B^n)}  
	+  \|D^{\kappa_1}(|\cdot|^2  f_1 f_2  f_3) \|_{L^2(\R^n\setminus \B^n)}.
	\end{align*}
	Now, also by Leibniz's rule we have
	\begin{align}\label{Eq:cubicterm}
	\|D^{k}  (|\cdot|^2  f_1 f_2  f_3 )  \|_{L^2(\R^n\setminus \B^n)}    \lesssim \sum_{r=0}^{2} \sum_{j_0+j_1+j_2+j_3= k-r} \||\cdot|^{2-r-j_0} \prod_{i=1}^3 D^{j_i}f_i \|_{L^2(\R^n\setminus \B^n)} .
	\end{align}
	In order to control the terms on the right we use Hardy's inequality \eqref{Eq:Hardy} and a generalized version of Strauss' inequality for higher homogeneous Sobolev spaces, 
	\begin{equation}\label{Eq:Strauss}
	\| |\cdot|^{\frac{n}{2}-j} f \|_{L^\infty(\R^n \backslash \B^n)}
	\lesssim \| D^jf \|_{L^2(\R^n)}
	\end{equation}
	for $1/2 < j < n/2,$ see e.g. \cite{ChoOza09}, Proposition 1. In Eq.~\eqref{Eq:cubicterm} we assume $0 \leq j_1 \leq j_2 \leq j_3 $ and we define distance functions
	\begin{equation}
	\overline{d}(j):=
	\begin{cases}
	\kappa_0-j & \text{ if } j<\kappa_0, \\
	\kappa_1-j & \text{ otherwise,} 
	\end{cases}
	\quad \text{and} \quad
	\underline{d}(j):=
	\begin{cases}
	\kappa_0-j & \text{ if } j \leq \kappa_0, \\
	\kappa_1-j & \text{ otherwise.} 
	\end{cases}
	\end{equation}
	Then, $k\in \{\kappa_0, \kappa_1\}$ and $n \geq 7$ imply  $2-r-j_0 \leq n-\overline{d}(j_1)-\overline{d}(j_2)-\underline{d}(j_3)$ and we therefore have 
	\begin{align*}
	&\| |\cdot|^{2-r-j_0} \prod_{i=1}^{3}D^{j_i} f_i \|_{L^2(\R^n\setminus \B^n)}  
	\leq \prod_{i=1}^{2}\| |\cdot|^{\frac{n}{2}-\overline{d}(j_i)}D^{j_i}f_i \|_{L^{\infty}(\R^n\setminus \B^n)} \:
	\: \| |\cdot|^{-\underline{d}(j_3)} D^{j_3} f_3   \|_{L^2(\R^n\setminus \B^n)}.
	\end{align*}
	The desired estimate then follows from Eqs.~\eqref{Eq:Strauss} and \eqref{Eq:Hardy}.
\end{proof}

From Lemma \eqref{Lem:Cubic} and the fact that $W(|
\cdot|) \in X$ we see that the nonlinear term \eqref{Eq:Nonlinear_term}, i.e. 
\begin{align}\label{Eq:Nonlinerity}
\mc N(f)(x)= - 3(d-2) \big (1+ |x|^2  W(|x|)\big ) f(x)^2 -  (d-2) |x|^2 f(x) ^3 
\end{align}
is a well-defined operator on $X$. In the following, we denote by $\mc B_X$ the unit ball in $X$.

\begin{lemma}\label{Le:Nonlinearity}
We have
\begin{align}\label{Eq:NonlinEst}
\| \mc N(f) - \mc N(g)\|_X \lesssim (\|f\|_X + \|g \|_X) \|f-g\|_X
\end{align}
for all $f,g \in \mc B_X$. Furthermore, $\mc N$ is differentiable at every $f \in X$ with Fr\'echet derivative $D\mc N(f): X \to X$ bounded. Furthermore, the mapping $f  \mapsto D \mc N(f)$ is continuous.
\end{lemma}

\begin{proof}

From Eq.~\eqref{Est:Cubic} we infer that for all $f,g \in X$,
 \[ \| \mc N(f)- \mc N(g) \|_X \leq \gamma(\|f\|_X,\|g\|_X) \| f- g \|_X \]
for a continuous function $\gamma: [0,\infty) \times [0,\infty) \to  [0,\infty)$ satisfying $\gamma(\|f\|_X,\|g\|_X) \lesssim \|f \|_X + \|g \|_X$ for all $f,g \in \mc B_X$. Furthermore, by using Eq.~\eqref{Est:Cubic} we simply conclude from the definition of the Fr\'echet derivative that
\[
D\mc N(f)g(x)=-3(d-2)f(x) \big (2+2|x|^2W(|x|)+|x|^2f(x) \big )g(x).
\]
Boundedness of $D\mc N(f)$  and continuity of $f \mapsto D\mc N(f)$ then easily follow.
 \end{proof}

\subsection{The initial data operator}
 To begin, we set $\mc R(v,T) := T v ( \sqrt{T} \cdot )$, and recall from Eq.~\eqref{Def:Initial_Data} that
\begin{align}\label{Def:InitialDataOperator}
 \mc U(v, T)  =  \mc R(v,T) + \mc R(W,T)-  \mc R(W,1)
\end{align}
In the rest of this section we prove basic mapping properties of the operator $\mc U$.

 \begin{lemma}\label{Le:InitialData}
The map $\mc U( v, T)$: $\mc B_X \times [\frac12, \frac32] \to X$ is continuous. Furthermore, if $\| v \|_X \leq \delta$ then
\[  \|\mc U(   v, T) \|_X \lesssim \delta \]
for all $T \in [1- \delta, 1 + \delta]$. 
\end{lemma}

\begin{proof}
First, note that for all $v_1,v_2 \in X$ and all $T \in [\frac12, \frac32]$ 
\begin{align*}
\| \mc R(v_1,T) - \mc R(v_2,T) \|_X \lesssim \|v_1  - v_2 \|_X
\end{align*}
i.e.~the map $\mc U(\cdot, T): X \to X$ is Lipschitz continuous. Next, for $v \in C^{\infty}_{\mathrm{rad}}(\R^n)$, and $T_1,T_2 \in [\frac12, \frac32]$ , the fundamental theorem of calculus implies that
\begin{align*}
v(\sqrt{T_1} \rho) - v(\sqrt{T_2} \rho) = (\sqrt{T_1}- \sqrt{T_2}) \int_0^{1} \rho v'(\rho(\lambda_1 - \lambda_2) s + r \lambda_2) ds.
\end{align*}
Now, the integral term can be controlled in $X$ provided that $v$ and its derivatives up to order $\kappa_1$ have sufficient decay at infinity. This in particular shows that 
\begin{align*}
\|  \mc R(W,T_1) -  \mc R(W,T_2) \|_X \lesssim |T_1 -  T_2|,
\end{align*}
i.e., $T \mapsto \mc R(W,T)$ is Lipschitz continuous. For general $v \in X$, this is not the case. However, for given $\tilde \varepsilon> 0$ we find a $\tilde v \in C_{\mathrm{c,rad}}^\infty(\R^n)$ with $\| v - \tilde v \|_X <  \tilde \varepsilon$ such that 
\begin{align*}
\| v(\sqrt{T_1} \cdot) - v(\sqrt{T_2} \cdot)  \|_X \leq  & \| v(\sqrt{T_1} \cdot) - \tilde v(\sqrt{T_1} \cdot)  \|_X + \| \tilde v(\sqrt{T_1} \cdot) - \tilde v(\sqrt{T_2} \cdot)  \|_X  \\
& + \| \tilde v(\sqrt{T_2} \cdot) -  v(\sqrt{T_2} \cdot)  \|_X  \lesssim  \tilde \varepsilon +   |T_1 -  T_2|.
\end{align*}
Hence, for given $(v_1,T_1) \in \mc B_X \times [\frac12,\frac32]$ and  $\varepsilon > 0$ let $(v_2,T_2)$ be such that $\|v_1 - v_2 \|_X  + |T_1 - T_2| < \delta$ for $\delta > 0$. Furthermore, chose  $\tilde v_1 \in   C_{\mathrm{c,rad}}^\infty(\R^n) $ such that  $\| v_1 - \tilde v \|_X < \delta$, then by the above considerations
\begin{align*}
\|  & \mc U(v_1,T_1)  - \mc U(v_2,T_2) \|_X  \leq  \| \mc R(v_1,T_1) -  \mc R(v_1,T_2) \|_X  \\
& +  \| \mc R(v_1,T_2) - \mc R(v_2,T_2) \|_X +  \|  \mc R(W,T_1) -  \mc R(W,T_2)\|_X    \lesssim  \delta.
\end{align*} 
This implies the claim provided that $\delta$ is chosen sufficiently small.  Finally, for $v \in X$, $\| v\|_X \leq \delta$
we get 
\begin{align*}
 \| \mc U(v, T ) \|_X &\lesssim \| v \|_X + |T - 1| \lesssim \delta 
\end{align*}
for all $T\in [1-\delta,1+\delta]$.
\end{proof}

\subsection{The nonlinear time evolution}

We consider the integral version of Eq.~\eqref{Eq:YM_AbstractPerturbation} by using the Duhamel formula and the above defined operators. 
 \begin{align}\label{Eq:IntegralEq}
 \Phi(\tau) = S(\tau) \mc U( v, T )  + \int_0^{\tau} S(\tau - \tau') \mc N ( \Phi(\tau')) d \tau'.
 \end{align}
Throughout this section, $\omega$ stands for the parameter given by Proposition \ref{Prop:DecayX}. The aim of this section is to prove the following result. 

\begin{theorem}\label{Th:MainSim}
Let $M > 0$ be sufficiently large and $\delta > 0$ sufficiently small. For every $v \in X$ with $\| v \|_X \leq \frac{\delta}{M^2}$,
there exists  a $T = T_{v} \in [1- \frac{\delta}{M}, 1 + \frac{\delta}{M}]$ and a unique function $\Phi  \in C([0,\infty), X)$ that satisfies
Eq.~\eqref{Eq:IntegralEq} for all $\tau \geq 0$. Furthermore, 
\[ \| \Phi (\tau) \|_X \leq \delta e^{-\omega \tau}, \quad \forall \tau \geq 0. \] 
\end{theorem}

First, we introduce the Banach space
\begin{align}
\mc X := \{ \Phi \in C([0,\infty), X) : \| \Phi \|_{\mc X} := \sup_{\tau \geq 0}  e^{\omega \tau } \|\Phi(\tau) \|_X   < \infty \}
\end{align}
and set $\mc X_{\delta} := \{ \Phi \in \mc X:  \| \Phi \|_{\mc X} \leq \delta \}$. To control the behavior of the semigroup on the unstable subspace $\mc P X$, we define the correction term
\begin{align}\label{Def:Correction}
\mc C(\Phi, u) := \mc P  u + \int_0^{\infty} e^{-\tau'} \mc P \mc N(\Phi(\tau')) d\tau'
\end{align}
and set 
\begin{align}\label{Def:ModInt}
K(\Phi,u)(\tau) := S(\tau)u + 	 \int_0^{\tau} S(\tau - \tau') \mc N ( \Phi(\tau')) d \tau' - e^{\tau} \mc C(\Phi, u).
\end{align}

\begin{lemma}\label{Le:ModIntegralOP}
There is $c > 0$ such that for all $\delta > 0$ sufficiently small and all $u \in X$ with $\|u \|_X \leq \frac{\delta}{c}$, the operator $K(\cdot,u)$ maps the ball $\mc X_{\delta}$ into itself. Furthermore 
\[ \| K(\Phi,u) - K(\Psi,u) \|_{\mc X} \leq \frac{1}{2} \| \Phi - \Psi \|_\mc X \]
for all $\Phi, \Psi \in \mc X_{\delta}$ and all $u \in X$. 
\end{lemma}

\begin{proof}
We have 
\[ \mc P K(\Phi,u)(\tau) =  - \int_{\tau}^{\infty} e^{-(\tau' - \tau)} \mc P \mc N(\Phi(\tau')) d\tau' \]
and 
\[ (1- \mc P) K(\Phi,u)(\tau) = S(\tau) (1- \mc P) u +  \int_{0}^{\tau} S(\tau - \tau')(1- \mc P )\mc N(\Phi(\tau')) d\tau' \]
From this it is straightforward to see that Lemma \ref{Le:Nonlinearity} implies 
\[  \| \mc P K(\Phi,u)(\tau) \|_X \lesssim e^{-2\omega \tau}  \delta^2, \]
and 
\[  \|( 1- \mc P) K(\Phi,u)(\tau) \|_X \lesssim e^{-\omega \tau} ( \tfrac{\delta}{c} + \delta^2) \]
for all $\Phi \in \mc X_{\delta}$ and $u \in X$ satisfying $\|u \|_X \leq \frac{\delta}{c}$. This implies the first claim. For the Lipschitz estimate we use that
\[  \| \mc N(\Phi(\tau))  -  \mc N(\Psi(\tau)) \|_X \lesssim \delta e^{-\omega \tau } \| \Phi - \Psi \|_{\mc X}  \]
by Lemma \ref{Le:Nonlinearity} to obtain
\[ \|K(\Phi,u)(\tau)-  K(\Psi,u)(\tau) \|_X \lesssim \delta e^{-\omega \tau }  \| \Phi - \Psi \|_{\mc X}  \]
which yields the result provided that $\delta >0$ is chosen sufficiently small. 
\end{proof}

\subsection{Proof of Theorem \ref{Th:MainSim}}
Let $v \in X$ with $\| v \|_X \leq \frac{\delta}{M^2}$.  By Lemma \ref{Le:InitialData} we can chose $M > 0$ large enough to guarantee that 
\[ \| \mc U(v,T) \|_X \leq \tfrac{\delta}{c} \]
for all $T \in  I_{\delta, M} := [1 -\frac{\delta}{M}, 1 + \frac{\delta}{M}]$, 
where $c > 0$ is the constant from Lemma \ref{Le:ModIntegralOP}. An application of the Banach fixed point theorem implies that for every $T \in  I_{\delta, M}$  there exists a unique solution $\Phi_{T} \in \mc X_{\delta}$ to the equation 
\begin{align}
\Phi(\tau) = K (\Phi, \mc U(v,T))(\tau), \quad \tau \geq 0.
\end{align}
Furthermore, by Lemma \ref{Le:InitialData} and continuity of the solution map, the map $T \mapsto \Phi_{T}$ is continuous. To prove Theorem \ref{Th:MainSim}, we show that there exists a $T = T(v)$ such that $\mc C(\Phi_{T(v)},  \mc U(v,T) ) = 0$. In fact, it is enough to show that
\begin{align}\label{Eq:VanishingCorr}
\left (\mc C(\Phi_{T(v)},  \mc U(v,T) )| \mb g \right )_{L^2_\sigma(\R^n)} = 0.
\end{align}
For this, we use that $\partial_T \mc R(W,T)|_{T=1} = \alpha \mb g$ for some $\alpha \in \R$ to write 
\[  \mc U(v,T)  = \mc R(v,T) + \alpha (T-1) \mb g + (T-1)^2 R(T,\cdot) \]
by Taylor expansion, where the error term depends continuously on $T$ and satisfies $\|R(T,\cdot)\|_X \lesssim 1$ for all $T \in  I_{\delta, M}$.
Thus,
\[ (\mc P  \mc U(v,T)|  \mb g  )_{L^2_\sigma(\R^n)} = C (1 - T) + f(T), \]
where $|f(T)| \lesssim \frac{\delta}{M^2}  + \delta^2 $.  By using the bounds of Lemma \ref{Le:Nonlinearity},  Eq.~\eqref{Eq:VanishingCorr} can be written as 
$T  = F(T) + 1$ for a continuous function $F$ that satisfies $|F(T)| \lesssim  \frac{\delta}{M^2}  + \delta^2$. Hence, by choosing $M > 0$ sufficiently large and $\delta = \delta(M) > 0$ sufficiently small 
 we obtain  $|F(T)| \leq \frac{\delta}{M}$, hence $T \mapsto F(T) +1:  I_{\delta, M} \to I_{\delta, M}$. An application of Brower's fixed point argument shows that there is a $T \in  I_{\delta, M}$ such that Eq.~\eqref{Eq:VanishingCorr} is satisfied and that the corresponding $\Phi_T \in \mc X_{\delta}$ solves Eq.~\eqref{Eq:IntegralEq}. For the proof of the uniqueness of the solution $\Phi_T$ in $C([0,\infty),X)$ we direct the reader to e.g. \cite{DonSch12}, Theorem 4.11.

\subsection{Theorem \ref{Th:MainSim} implies Theorem \ref{Thm:Main}}

Fix $5 \leq d \leq 9$ and set $n = d+2$. Let $\delta > 0$ and $M > 0$ be such that Theorem \ref{Th:MainSim} holds and set $\delta' := \frac{\delta}{M}$.
Let $v \in  C^{\infty}_{c,{\mathrm{rad}}}(\R^{n})$ such that $\|v \|_X \leq \frac{\delta'}{M}$.  Then there exists a function $\Phi \in \mc X_{\delta}$ and a $T \in [1-\delta, 1+\delta]$ such that \eqref{Eq:IntegralEq} is satisfied for all $\tau \geq 0$. Our assumption on the data imply that $\mc U(v, T) \in \mc D(\mc L_X)$ and thus, in view of differentiability properties of the operator $\mc N$ (see Lemma \ref{Le:Nonlinearity}), $\Phi \in C^1((0,\infty), X)$, $\Phi(\tau) \in  \mc D(\mc L_X)$ for all $\tau \geq 0$ and 
\begin{equation}\label{Eq:Phi(tau)}
 \partial_{\tau} \Phi(\tau) = \mc L_X \Phi(\tau) + \mc N(\Phi(\tau)) \quad \forall \tau>0,
 \end{equation}
 see e.g. \cite{Paz83}, Theorem 6.1.5. 
To show smoothness of $\Phi(\tau)$ we do the following. First, remember that $\mc L_X=\mc L_0|_X + L'$ where $L'$ is a bounded operator on $X$. Hence, since $\Phi$ satisfies Eq.~\eqref{Eq:Phi(tau)} we have
\begin{equation}
	\Phi(\tau)=S_0(\tau)\Phi(0)+\int_{0}^{\tau}S_0(\tau-s) (L' \Phi(s) + \mc N(\Phi(s)))ds
\end{equation}  
and $\Phi(\tau) \in C^{\infty}(\R^n)$ follows from the smoothing properties of $S_0(\tau)$, which are obvious from its explicit form, see Lemma \ref{Le:L0sa}. By setting $\psi(\tau,\cdot) := W(|\cdot|) + \Phi(\tau)$ we obtain a solution to the initial value problem \eqref{Eq:SelfSim}-\eqref{Eq:SelfSimData}. Finally, by translating back to the original coordinates we see that
\begin{align*}
u(r,t) =  \frac{1}{T-t} \psi( \tfrac{r}{\sqrt{T-t}}, - \log(T-t) + \log T) 
\end{align*}
is the unique classical solution to  Eq.~\eqref{Eq:EquivarEq} which evolves from the initial data \eqref{Eq:InitialData} and belongs to $C([0,\infty),X) \cap C^1((0,\infty),X)$.  From Eq.~\eqref{Eq:ScalinguT}
and
\begin{align*}
  (T-t)^{ \frac12 (\kappa_1 + 2 - \frac{n}{2} )} & \| u(|\cdot|,t)  -u_{T}(|\cdot|,t) \|_X  \\& =   (T-t)^{\frac{\kappa_1}{2}  - \frac{n}{4}}   \|\Phi(- \log(T-t) + \log T)(\tfrac{|\cdot|}{\sqrt{T-t}}) \|_X  \\& \lesssim \|\Phi(- \log(T-t) + \log T)) \|_X  \leq \delta (T-t)^{\omega}
\end{align*}
we derive Eq.~\eqref{Eq:X_estimate}. Convergence in $L^{\infty}$ follows from Eq.~\eqref{Eq:L_inf}, namely
\begin{multline*}
\frac{\| u(\cdot,t)  -  u_T(\cdot,t) \|_{L^{\infty}(\R^+)}}{\| u_T(\cdot,t) \|_{L^{\infty}(\R^+)}} \simeq  (T-t)\| u(\cdot,t)    -  u_T(\cdot,t) \|_{L^{\infty}(\R^+)} \\
 \simeq  \|\Phi(- \log(T-t) + \log T)) \|_{L^{\infty}(\R^n)}  
 \lesssim  \|\Phi(- \log(T-t) + \log T)) \|_X  \leq \delta (T-t)^{\omega}.
\end{multline*}
\appendix

\section{A variant of GGMT bound}\label{Sec:GGMT}

In this section, we establish a criterion for the absence of negative spectrum for a class of self-adjoint operators. Our result is in the spirit of the so-called GGM(T) bounds for the number of bound states of higher angular momenta for radial Schr\"odinger operators, see \cite{GlaGroMarThi76,GlaGroMar78}. Our setting is the following. Let $A$ be a linear operator defined on $\mc D(A) = C_c^\infty(\R^+)$ by
\begin{equation}\label{Def:A}
	Au(\rho):=-u''(\rho)+\frac{\alpha}{\rho^2}u(\rho)+V(\rho)u(\rho),
\end{equation}
where $\alpha > -\frac{1}{4}$ and $V\in L_\text{loc}^2(\R^+)$ is a real valued function. Clearly, $A$  is a symmetric operator on $L^2(\R^+)$. What is more, we have the following result.  
\begin{theorem}\label{Thm:GGMT}
	Assume that the potential $V$ is such that the closure of $A$ (which we denote by $\mc A$) is self-adjoint. Furthermore, write $V=V_+-V_-$ where $V_{\pm}\geq0$, and let $p\geq1$. If
	\begin{equation}\label{Eq:Assum}
	\int_{\R^+}\rho^{2p-1}|V_-(\rho)|^p d\rho<\frac{(4\alpha+1)^{\frac{2p-1}{2}}p^p \Gamma(p)^2}{(p-1)^{p-1}\Gamma(2p)}
	\end{equation} 
	then $\sigma(\mc A) \subseteq [0,+\infty)$. Furthermore, zero is not an eigenvalue of $\mc A$.
\end{theorem}
\begin{proof}
	We argue by contradiction. Suppose that there is a negative spectral point of $\mc A$. Then 
	\[
	\inf_{u\in\mc D(\mc A),\,\|u\|_{L^2(\R^+)}=1}( \mc Au|u)_{L^2(\R^+)} = \inf \sigma(\mc A)<0,
	\]
	and due to closedness of $\mc A$ and the continuity of the inner product, there exists $v\in C^\infty_c(\R^+)$ such that $( \mc Av|v )_{L^2(\R^+)}<0$. This in turn implies that for $p\geq1$
	\begin{align*}
		\int_0^{+\infty}\left(v'(\rho)^2+\alpha \rho^{-2}v(\rho)^2\right)d\rho& < \int_{\R^+}V_-(\rho)v(\rho)^2d\rho\\ 
		&\leq \|(\cdot)^{\frac{2p-1}{p}}V_- \|_{L^p(\R^+)} \|(\cdot)^{-\frac{2p-1}{p}}v^2 \|_{L^{\frac{p}{p-1}}(\R^+)},
	\end{align*}
	by integration by parts and H\"older's inequality. In particular, 
	\begin{equation}\label{Eq:Inf}
			\left(\int_{\R^+}\rho^{2p-1}|V_-(\rho)|^p d\rho\right)^\frac{1}{p} \geq  \inf_{v\in C^\infty_c(\R^+), \, v \neq 0}
			\frac{\int_0^{+\infty}\left(v'(\rho)^2+\alpha \rho^{-2}   v(\rho)^2 \right)d\rho}{\left(\int_{0}^{+\infty}\rho^{-\frac{2p-1}{p-1}}|v(\rho)|^{\frac{2p}{p-1}} d\rho \right)^\frac{p-1}{p}}:=\mu(p).
	\end{equation}
	To compute $\mu(p)$ we introduce the following change of variables $ \rho=e^{\frac{x}{\sqrt{4\alpha+1}}}$, $u(\rho)\sqrt{\rho}=\phi(x)$. This leads to
	\begin{equation*}
		\mu(p)=(4\alpha+1)^{\frac{2p-1}{2p}}\inf_{\phi\in C^\infty_c(\R),\, \phi\neq0}\frac{\int_{-\infty}^{+\infty}(\phi'(x)^2+\frac{1}{4}\phi(x)^2)dx}{\left(\int_{-\infty}^{+\infty}|\phi(x)|^{\frac{2p}{p-1}}dx\right)^\frac{p-1}{p}}.
	\end{equation*}
	The value of the infimum on the right is known, see e.g. \cite{CreDonSchSne17}, Lemma A.2, and this gives $$\mu(p)=(4\alpha+1)^{\frac{2p-1}{2p}}\frac{p}{p-1}\left(\frac{(p-1) \Gamma(p)^2}{\Gamma(2p)}\right)^{\frac{1}{p}}.$$
	Finally, Eq.~\eqref{Eq:Inf} is in contradiction with Eq.~\eqref{Eq:Assum}. It remains to prove that zero is not an eigenvalue. We argue by contradiction again. Assume that there is a nontrivial $\psi\in \mc D(\mc A)$  for which $\mc A\psi=0$. This implies $(\mc A\psi|\psi)_{L^2(\R^+)}=0$. Then for $\mc A_\varepsilon$,  which is obtained from $\mc A$ by replacing $V(\rho)$ with $V_\varepsilon(\rho):=V(\rho)-\varepsilon e^{-\rho}$, we see that for small enough $\varepsilon>0$ the inequality \eqref{Eq:Assum} holds and at the same time
	$$(\mc A_\varepsilon \psi|\psi)_{L^2(\R^+)}<(\mc A\psi|\psi)_{L^2(\R^+)}=0,$$
	which is a contradiction based on the first part of the proof.
\end{proof}

\pagestyle{plain}

\begin{thebibliography}{10}

\bibitem{BieBiz11}
Pawe{\l} Biernat and Piotr Bizo\'{n}.
\newblock Shrinkers, expanders, and the unique continuation beyond generic
  blowup in the heat flow for harmonic maps between spheres.
\newblock {\em Nonlinearity}, 24(8):2211--2228, 2011.

\bibitem{BieDon18}
Pawe{\l} Biernat and Roland Donninger.
\newblock Construction of a spectrally stable self-similar blowup solution to
  the supercritical corotational harmonic map heat flow.
\newblock {\em Nonlinearity}, 31(8):3543--3566, 2018.

\bibitem{BieDonSch17}
Pawe{\l} Biernat, Roland Donninger, and Birgit Sch\"{o}rkhuber.
\newblock Stable self-similar blowup in the supercritical heat flow of harmonic
  maps.
\newblock {\em Calc. Var. Partial Differential Equations}, 56(6):Art. 171, 31,
  2017.

\bibitem{BieSek19}
Pawe{\l} Biernat and Yukihiro Seki.
\newblock Type {II} blow-up mechanism for supercritical harmonic map heat flow.
\newblock {\em Int. Math. Res. Not. IMRN}, (2):407--456, 2019.

\bibitem{BizWas15}
Piotr Bizo\'{n} and Arthur Wasserman.
\newblock Nonexistence of shrinkers for the harmonic map flow in higher
  dimensions.
\newblock {\em Int. Math. Res. Not. IMRN}, (17):7757--7762, 2015.

\bibitem{CheZha15}
Zhengxiang Chen and Yongbing Zhang.
\newblock Stabilities of homothetically shrinking {Y}ang-{M}ills solitons.
\newblock {\em Trans. Amer. Math. Soc.}, 367(7):5015--5041, 2015.

\bibitem{ChoOza09}
Yonggeun Cho and Tohru Ozawa.
\newblock Sobolev inequalities with symmetry.
\newblock {\em Commun. Contemp. Math.}, 11(3):355--365, 2009.

\bibitem{CosDonGlo17}
Ovidiu Costin, Roland Donninger, and Irfan Glogi{\'{c}}.
\newblock Mode stability of self-similar wave maps in higher dimensions.
\newblock {\em Comm. Math. Phys.}, 351(3):959--972, oct 2017.

\bibitem{CreDonSchSne17}
Matthew Creek, Roland Donninger, Wilhelm Schlag, and Stanley Snelson.
\newblock Linear stability of the skyrmion.
\newblock {\em Int. Math. Res. Not. IMRN}, (8):2497--2537, 2017.

\bibitem{DonSch12}
Roland Donninger and Birgit Sch{\"o}rkhuber.
\newblock Stable self-similar blow up for energy subcritical wave equations.
\newblock {\em Dyn. Partial Differ. Equ.}, 9(1):63--87, 2012.

\bibitem{DonSch19}
Roland Donninger and Birgit Sch\"{o}rkhuber.
\newblock Stable blowup for the supercritical {Y}ang--{M}ills heat flow.
\newblock {\em J. Differential Geom.}, 113(1):55--94, 2019.

\bibitem{EngNag00}
Klaus-Jochen Engel and Rainer Nagel.
\newblock {\em One-parameter semigroups for linear evolution equations}, volume
  194 of {\em Graduate Texts in Mathematics}.
\newblock Springer-Verlag, New York, 2000.
\newblock With contributions by S. Brendle, M. Campiti, T. Hahn, G. Metafune,
  G. Nickel, D. Pallara, C. Perazzoli, A. Rhandi, S. Romanelli and R.
  Schnaubelt.

\bibitem{Fan99}
Huijun Fan.
\newblock Existence of the self-similar solutions in the heat flow of harmonic
  maps.
\newblock {\em Sci. China Ser. A}, 42(2):113--132, 1999.

\bibitem{Gas02}
Andreas Gastel.
\newblock Singularities of first kind in the harmonic map and {Y}ang-{M}ills
  heat flows.
\newblock {\em Math. Z.}, 242(1):47--62, 2002.

\bibitem{Gho17}
Tej-eddine {Ghoul}.
\newblock {Stable type II blowup for the 7 dimensional 1-corotational energy
  supercritical harmonic map heat flow}.
\newblock {\em arXiv e-prints}, page arXiv:1710.09293, Oct 2017.

\bibitem{GhoNguTie19}
Tej-eddine Ghoul, Slim Ibrahim, and Van~Tien Nguyen.
\newblock On the stability of type ii blowup for the 1-corotational
  energy-supercritical harmonic heat flow.
\newblock {\em Anal. PDE}, 12(1):113--187, 2019.

\bibitem{GlaGroMar78}
V.~Glaser, H.~Grosse, and A.~Martin.
\newblock Bounds on the number of eigenvalues of the {S}chr\"{o}dinger
  operator.
\newblock {\em Comm. Math. Phys.}, 59(2):197--212, 1978.

\bibitem{GlaGroMarThi76}
V.~Glaser, H.~Grosse, A.~Martin, and W.~Thirring.
\newblock A family of optimal conditions for the absence of bound states in a
  potential.
\newblock {\em Les rencontres physiciens-math\'ematiciens de Strasbourg
  -RCP25}, 23, 1976.

\bibitem{Gro01}
Joseph~F. Grotowski.
\newblock Finite time blow-up for the {Y}ang-{M}ills heat flow in higher
  dimensions.
\newblock {\em Math. Z.}, 237(2):321--333, 2001.

\bibitem{GroSha07}
Joseph~F. Grotowski and Jalal Shatah.
\newblock Geometric evolution equations in critical dimensions.
\newblock {\em Calc. Var. Partial Differential Equations}, 30(4):499--512,
  2007.

\bibitem{Kat95}
Tosio Kato.
\newblock {\em Perturbation theory for linear operators}.
\newblock Classics in Mathematics. Springer-Verlag, Berlin, 1995.
\newblock Reprint of the 1980 edition.

\bibitem{KelStr16}
Casey Kelleher and Jeffrey Streets.
\newblock Entropy, stability, and {Y}ang-{M}ills flow.
\newblock {\em Commun. Contemp. Math.}, 18(2):1550032, 51, 2016.

\bibitem{KelleherStreets2018}
Casey Kelleher and Jeffrey Streets.
\newblock Singularity formation of the {Y}ang-{M}ills flow.
\newblock {\em Ann. Inst. H. Poincar\'{e} Anal. Non Lin\'{e}aire},
  35(6):1655--1686, 2018.

\bibitem{KelleherStreets2018c}
Casey Kelleher and Jeffrey Streets.
\newblock Corrigendum to ``{S}ingularity formation of the {Y}ang-{M}ills flow''
  [{A}nn. {I}. {H}. {P}oincar\'{e}---{AN} 35 (6) (2018) 1655--1686].
\newblock {\em Ann. Inst. H. Poincar\'{e} Anal. Non Lin\'{e}aire},
  36(5):1501--1502, 2019.

\bibitem{MusSch13}
Camil Muscalu and Wilhelm Schlag.
\newblock {\em Classical and multilinear harmonic analysis. {V}ol. {I}}, volume
  137 of {\em Cambridge Studies in Advanced Mathematics}.
\newblock Cambridge University Press, Cambridge, 2013.

\bibitem{Nai94}
Hisashi Naito.
\newblock Finite time blowing-up for the {Y}ang-{M}ills gradient flow in higher
  dimensions.
\newblock {\em Hokkaido Math. J.}, 23(3):451--464, 1994.

\bibitem{Paz83}
A.~Pazy.
\newblock {\em Semigroups of linear operators and applications to partial
  differential equations}, volume~44 of {\em Applied Mathematical Sciences}.
\newblock Springer-Verlag, New York, 1983.

\bibitem{RapSch13}
Pierre Rapha{\"e}l and Remi Schweyer.
\newblock Stable blowup dynamics for the 1-corotational energy critical
  harmonic heat flow.
\newblock {\em Comm. Pure Appl. Math.}, 66(3):414--480, 2013.

\bibitem{RapSch14}
Pierre Rapha{\"e}l and Remi Schweyer.
\newblock Quantized slow blow-up dynamics for the corotational energy-critical
  harmonic heat flow.
\newblock {\em Anal. PDE}, 7(8):1713--1805, 2014.

\bibitem{ReeSim75}
Michael Reed and Barry Simon.
\newblock {\em Methods of modern mathematical physics. {II}. {F}ourier
  analysis, self-adjointness}.
\newblock Academic Press [Harcourt Brace Jovanovich, Publishers], New
  York-London, 1975.

\bibitem{SchStrTah98}
Andreas~E. Schlatter, Michael Struwe, and A.~Shadi Tahvildar-Zadeh.
\newblock Global existence of the equivariant {Y}ang-{M}ills heat flow in four
  space dimensions.
\newblock {\em Amer. J. Math.}, 120(1):117--128, 1998.

\bibitem{Tes09}
Gerald Teschl.
\newblock {\em Mathematical methods in quantum mechanics}, volume~99 of {\em
  Graduate Studies in Mathematics}.
\newblock American Mathematical Society, Providence, RI, 2009.

\bibitem{Wal19}
Alex Waldron.
\newblock Long-time existence for {Y}ang-{M}ills flow.
\newblock {\em Invent. Math.}, 217(3):1069--1147, 2019.

\bibitem{Wei87}
Joachim Weidmann.
\newblock {\em Spectral theory of ordinary differential operators}, volume 1258
  of {\em Lecture Notes in Mathematics}.
\newblock Springer-Verlag, Berlin, 1987.

\bibitem{Wei04}
Ben Weinkove.
\newblock Singularity formation in the {Y}ang-{M}ills flow.
\newblock {\em Calc. Var. Partial Differential Equations}, 19(2):211--220,
  2004.

\end{thebibliography}

\bibliographystyle{plain}

\end{document}